\documentclass[12pt]{amsart}
\usepackage{amssymb,latexsym,eufrak,amsmath,amscd, graphicx}
\usepackage[colorlinks=true,pagebackref,hyperindex]{hyperref}
\usepackage{amsthm}
\usepackage{amsfonts}
\usepackage{amscd}

\renewcommand{\phi}{\varphi}
\renewcommand{\epsilon}{\varepsilon}
\renewcommand{\theta}{\vartheta}

\def\ZZ{{\mathbf Z}}

\def\PP{{\mathbf P}}

\def\cE{\mathcal{E}}

\def\cL{\mathcal{L}}
\def\cO{\mathcal{O}}
\def\cT{\mathcal{T}}
\def\cP{\mathcal{P}}
\def\cR{\mathcal{R}}

\def\alb{{\alpha\beta}}

\DeclareMathOperator{\codim}{codim} \DeclareMathOperator{\Pic}{Pic}
\DeclareMathOperator{\Hom}{Hom}

\DeclareMathOperator{\rank}{rank} \DeclareMathOperator{\Spec}{Spec}
\DeclareMathOperator{\Sch}{Sch}
\DeclareMathOperator{\Set}{Set}
\DeclareMathOperator{\Ann}{Ann}
\DeclareMathOperator{\id}{id}
 
 \DeclareMathOperator{\lct}{lct}

\DeclareMathOperator{\pic}{Pic}\DeclareMathOperator{\mult}{mult}

\newtheorem{lemma}{Lemma}[section]
\newtheorem{theorem}[lemma]{Theorem}
\newtheorem*{thma}{Theorem A}
\newtheorem*{thmb}{Theorem B}

\newtheorem{corollary}[lemma]{Corollary}

\theoremstyle{definition}
\newtheorem{definition}[lemma]{Definition}
\newtheorem{remark}[lemma]{Remark}

\theoremstyle{remark}
\newtheorem*{remark*}{Remark}
\newtheorem*{note*}{Note}

\begin{document}

\title{Jet schemes and singularities of $W^r_d(C)$ loci}

\thanks{2000\,\emph{Mathematics Subject Classification}. Primary 14E18; Secondary 14B05.}

\author[Z. Zhu]{Zhixian Zhu}
\address{Department of Mathematics, University of Michigan,
Ann Arbor, MI 48109, USA} \email{{\tt zhixian@umich.edu}}

\begin{abstract}
Kempf proved that  the theta divisor of a smooth projective curve $C$ has rational singularities. In this paper we estimate the dimensions of the jet schemes of the theta divisor and show that all these schemes are irreducible. In particular, we recover Kempf's theorem in this way.   For general projective smooth curves, our method also gives a formula for the log canonical threshold of the pair $(\pic^d(C), W^r_d(C))$.
\end{abstract}

\maketitle

\markboth{Z. Zhu}{Jet schemes and singularities of $W^r_d(C)$ loci}

\section*{Introduction}

Let $k$ be an algebraically closed field of characteristic zero. Let $C$ be a smooth projective curve of genus $g$ over $k$. Riemann's original problem was to determine the order of vanishing of the theta function at a point in the Jacobian of $C$. His Singularity Theorem  says that for every line bundle $L$ of degree $g-1$ in the  theta divisor $\Theta$,  the multiplicity of $\Theta$ at $L$ is $h^0(C,L)$.

Recall that $W^r_d(C)$ is the subscheme of $\pic^d(C)$ parameterizing line bundles $L$ of degree $d$ with $\dim|L|\geq r$. The theta divisor $\Theta$ is $W^0_{g-1}(C)$.  Kempf \cite{Kem} described the tangent cone of $W^0_d$ at every point. In particular, he generalized  Riemann's multiplicity result to the $W^0_d$  locus. In his paper, he described the singularities of $W^0_d$ and its tangent cone as follows.
Let $L$ be a point of $W^0_d(C)$,  with $d<g$ and $l=\dim H^0(L)$. The tangent cone $\cT_L(W^0_d(C))$ has rational singularities and therefore $W^0_d(C)$ has rational singularities. Moreover, the degree of $\PP \cT_L(W^0_d(C))$ as a subscheme of $\PP H^0(C,K)^*$ is the binomial coefficient $${h^1(L) \choose {l-1}}={ {g-d
+l-1}\choose {l-1}}.$$

Following the work of Riemann and Kempf, there has been much interest in the singularities of general theta divisors.  For instance, using vanishing theorems,  Ein and
Lazarsfeld \cite{EL} showed that if $\Theta\subset A$ is an irreducible theta divisor on an abelian variety, then $\Theta$ is normal and has rational singularities.

In this paper we approach the study of the singularities of the Brill-Noether locus $W^r_d(C)$ from the point of view of its jet schemes. The jet scheme $X_m$ of a given scheme $X$ of finite type over $k$ parameterizes $m$-jets on X, that is, morphisms $\Spec k[t]/(t^{m+1}) \rightarrow X$. Note that $X_0=X$ and $X_1$ is the total tangent space of $X$.  For every $m\geq 0$, we have a morphism $\pi_m: X_m\rightarrow X$ that maps an $m$-jet to the image of the closed point. The fiber of $\pi_m$ at $x\in X$ is denoted by $X_{m,x}$.

Let $L$ be a point in $\Theta$. By the definition of $\pic^d(C)$,  an element $\cL_m\in \pic^d(C)_m$ is identified with a line bundle on $C\times \Spec k[t]/t^{(m+1)}$.

Using the description of the theta divisor as a determinantal variety, we partition the scheme $\Theta_{m,L}$ into constructible subsets $C_{\lambda}$ indexed by partitions $\lambda$ of length $h^0(C,L)$ with sum $\geq m+1$. Several invariants of $\cL_m \in \Theta_{m,L}$ are determined by the corresponding partition $\lambda$. For instance, $\lambda$ determines the dimension of the kernel of the truncation map $H^0(C\times\Spec k[t]/t^{(m+1)}, \cL_m)\rightarrow H^0(C,L)$. In this way, $\lambda$  determines for each $j\leq m$ the dimension of the subspace of sections in $H^0(L)$ that can be extended to sections of $\cL_j$, where $\cL_j$ is the image of $\cL_m$ under the truncation map $\pic^d(C)_m\rightarrow \pic^d(C)_j$.

Let us briefly describe the proof of Riemann's Singularity Theorem that we give using jet schemes. Since the inequality $\mult_{L}{\Theta}\geq l:=h^0{(L)}$ follows from the determinantal description of $\Theta$, we focus on the opposite inequality.  In order to show that $\mult_{L}{\Theta}\leq l$, it is enough to prove that $\Theta_{l,L}\neq \pic^{g-1}(C)_{l,L}$. If this is not the case, then the image of $\Theta_{l,L}$ in $\Theta_{1,L}=\pic^{g-1}(C)_{1,L}$ is $\Theta_{1,L}$. Using the partition associated to any $\cL_l\in \Theta_{l,L}$, we show that if $\cL_1$ is the image of $\cL_l$ in $\Theta_{1,L}$, then the restriction map $H^0(\cL_1)\rightarrow H^0(L)$ is nonzero.  On the other hand, we can identify  $\cL_1$ in $\pic^{g-1}(C)_{1,L}$ to a {\u C}ech cohomolgy class in $H^1(C,\cO_C)$. Furthermore, the obstruction to lifting a section $s\in H^0(C,L)$ to a section of  $\cL_1$ can be described using the pairing
$$H^0(C,L)\otimes H^1(C,\cO_C)\xrightarrow{\nu} H^1(C,L)$$
that is, s lifts if and only if $\nu(s\otimes \cL_1)=0$. Since the set of elements in $H^1(C,\cO_C)$ for which there is a nonzero such $s$ is of codimension one, this gives a contradiction, proving that $\mult_{L}\Theta\leq l$.

By estimating the dimension of the constructible subset $C_{\lambda}$ of $\Theta_{m,L}$ for each partition $\lambda$, we obtain the following result.

\begin{thma}\label{theta divisor}
For every smooth projective curve $C$ of genus $g\geq 3$ over $k$, and every integer $m\geq 1$, we have
$\dim(\pi^{\Theta}_m)^{-1}(\Theta_{\text{sing}})=(g-1)(m+1)-1$ if $C$ is a hyperelliptic curve. For nonhyperelliptic curves, we have $\dim(\pi^{\Theta}_m)^{-1}(\Theta_{\text{sing}})=(g-1)(m+1)-2$.
\end{thma}

A theorem in \cite{Mus1} describes complete intersection rational singularities in terms of jet schemes. Applying that theorem, we recover Kempf's result that the theta divisor has rational singularities.  Similarly, as a corollary of  Theorem 3.3 in \cite{EMY}, we deduce that the theta divisor of a nonhyperelliptic curve has terminal singularities.

Using similar ideas, we are able to compute the dimensions of the jet schemes of the Brill-Noether locus $W^r_d(C)$ for generic curves. Using Musta\c{t}\v{a}'s formula from \cite{Mmus2} describing the log canonical threshold in terms of dimensions of jet schemes, we obtain the following formula for the log canonical threshold of the pair $(\pic^d(C),W^r_d(C))$.

\begin{thmb} \label{lct of BN}
For a general projective smooth curve $C$ of genus $g$, let $L$ be a line bundle of degree $d$
with $d\leq g-1$ and $l=h^0(L)$. The log canonical threshold of $(\pic^d(C),W^r_d(C))$ at $L \in W^r_d(C)$  is
$$\lct_L(\pic^d(C), W^r_d(C))=\min\limits_{1\leq i\leq {l-r}}\left\{\frac{(l+1-i)(g-d+l-i)}{l+1-r-i}\right\}.$$
\end{thmb}

Recall that one can locally define  a map from $\pic^d(C)$ to a matrix space such that $W^r_d(C)$ is the pull back of a suitable generic determinantal variety. It follows from the above theorem that for generic curves, the local log canonical threshold of $(\pic^d(C),W^r_d(C))$ at $L$ is equal to the local log
canonical threshold of that generic determinantal variety at the image of $L$ (for the formula for the log canonical threshold of a generic determinantal variety, see Theorem 3.5.7. in \cite{Doc}).

In light of Ein-Lazarsfeld's result for general principal polarizations on abelian varieties, it would be interesting to understand  in general the properties of the jet schemes of such divisors.  We also hope that in the future one could use the behavior of the jet schemes of the Brill-Noether loci to distinguish other geometric properties of curves.

The paper is organized as follows. In the first section, we review some basic definitions and notation related to the jet schemes and the Brill-Noether loci. We give a criterion on the partition $\lambda$ associated to $\cL_m\in \pic^{g-1}(C)_{m}$ to have $\cL_m$ $\in\Theta_m$. We also prove the formula for  $h^0(\cL_m)$ in terms of the partition $\lambda$ associated to $\cL_m$. As a first application, we give the proof of Riemann's multiplicity formula that we sketched above.  In the second section, we prove Theorems A and B by estimating the dimensions of $W^r_d(C)_{m,L}$ for every $L\in W^r_d(C)$.

\section*{Acknowledgement}

It is a pleasure to thank Mircea Musta\c{t}\v{a} for introducing me to jet schemes and encouragement to start this project. I am grateful to Jesse Kass for his help and discussions Jacobians of curves. I would like to thank Linquan Ma and Sijun Liu for useful discussions in the earlier stages of this project.

\section{Introduction to jet schemes and varieties of special linear series on a curve}

Let $k$ be an algebraically closed field of characteristic zero. Given a scheme $X$ (of finite type) over $k$ and an integer $m\geq 0$, the jet scheme $X_m$  of order $m$ of $X$ is a scheme of finite type over $k$ satisfying the following
adjunction
\begin{equation}\label{adjunction}\tag{$\ddag$}
\Hom_{\Sch/k}(Y,X_m)\cong\Hom_{\Sch/k}(Y\times \Spec k[t]/(t^{m+1}), X)
\end{equation}
for every scheme $Y$ of  finite type over $k$.

In particular, the $k$-rational points of $X_m$ are identified with the $m$-jets of $X$, that is, with the morphisms $\Spec k[t]/(t^{m+1})\rightarrow X$. For every $j$ with $0\leq j\leq m$, the natural ring homomorphism $k[t]/(t^{m+1})\rightarrow  k[t]/(t^{j+1})$ induces a closed embedding \mbox{$\Spec k[t]/(t^{j+1})\rightarrow \Spec k[t]/(t^{m+1})$}.  The above adjunction induces a truncation map
\begin{equation*}
  \rho^m_j: ~X_m\rightarrow X_j\,.
\end{equation*}
For simplicity, we usually write $\pi_m^X$ or $\pi_m$ to denote the projection $\rho^m_0:X_m\rightarrow X$. For every fixed point $x\in X$, we write $X_{m,x}$ for the fiber of $\pi_m$ at $x$, the $m$-jets of $X$ centered at $x$. It turns out that the geometry of the jet schemes $X_m$ is closely related with the geometry of scheme $X$ itself.

Let $C$ be a smooth projective curve of genus $g$ over field $k$. We now recall the definition of $\pic^d(C)$. For every scheme $S$, let $p$ and $q$ be the projections of $S\times C$ onto $S$ and $C$ respectively.
{\it A family of degree $d$ line bundles on $C$ parameterized by a scheme $S$} is a line bundle on $C\times S$ which restricts to a degree
$d$ line bundle on $C\times\{s\}$, for every $s$ in $S$. We say that two such families $\cL$ and $\cL'$ are {\it equivalent} if
there is a line bundle $\cR$ on $S$ such that $\cL'\cong \cL\otimes q^*\cR$. $\pic^d(C)$ parameterizes degree $d$ line bundles on $C$; more  precisely, it represents the functor
$$F: \Sch/k \rightarrow \Set$$
where $F(S)$ is the set of equivalence classes of families of degree $d$ line bundles on $C$ parameterized by $S$. A universal line bundle $\cP$ on $C\times \pic^d(C)$ is a {\it Poincar\'e line bundle} of degree $d$ for $C$.

Recall now that $W^r_d(C)$ is the closed subset of $\pic^d(C)$ parameterizing line bundles $L$ of degree $d$ with $\dim |L|\geq r$:
$$W^r_d(C)=\{L \in \pic^{d}(C): \deg L=d, h^0(L)\geq r+1\}.$$
In particular, we have the theta divisor
$\Theta: =\{L \in \pic^{g-1}(C): h^0(C, L) \neq 0\}=W^0_{g-1}(C)$. Each $W^r_d(C)$ has a natural scheme structure as a degeneracy locus we now describe. \\

Let $E$ be any effective divisor on $C$ of degree $e\geq 2g-d-1$ and let $\cE=\cO(E)$.
The following facts are standard (see \cite[\S IV.3]{ACGH}).

For every family of degree $d$ line bundles $\cL$ on $S\times C$, the sheaves
${p}_*{(\cL\otimes q^*(\cE))}$ and ${p}_*(\cL\otimes q^*(\cE)\otimes\cO_{q^{-1}E})$ are locally free of
ranks $d+e+1-g$ and $e$, respectively. Moreover, there is an exact sequence on $S$
\begin{equation}\label{ses}
0\rightarrow {p}_*\cL\rightarrow {p}_*(\cL\otimes q^*(\cE))\xrightarrow{\Phi_\cL} {p}_*(\cL\otimes q^*(\cE)\otimes\cO_{q^
{-1}E})\rightarrow R^1{p}_*(\cL)\rightarrow 0.
\end{equation}

With the above notation, $W^r_d(C)$ represents the functor $\Sch/k \rightarrow \Set$ given by
\begin{equation*}
S \mapsto
\left\{
\begin{array}{c}
\text{equivalence classes of families } \cL \text{ of degree } d \text{ line bundles
on } \\
S\times C\xrightarrow{p} S  \text{ such that }  \rank(\Phi_{\cL})\leq d+e-g-r
\end{array}
\right \}.
\end{equation*}
It can be shown that the above condition $\rank(\Phi_{\cL})\leq d+e-g-r$ does not depend on the particular choice of $e$ and $E$.

In particular, the line bundle $L\in \pic^d(C)$ is in $W^r_d(C)$ if and only if  locally all the $e+d+1-g-r$
minors of $\Phi_L$ vanish.  Therefore $W^r_d(C)$ is a determinantal variety.

Let $T_m$ be the scheme $\Spec k[t]/(t^{m+1})$. We now discuss the jet schemes of the theta divisor $\Theta_m$ for all $m$. By the definition of
$\Theta$, we have $\Theta_m$ consists of line bundles $\cL_m \in \Pic(T_m\times C)$ such that $\deg(\cL_m|_{\{0\}\times C})=g-1$ and $\det(\Phi_{\cL_m})=0$ in $k[t]/(t^{m+1})$.

Given a positive integer $n$, we recall that a {\it partition of $n$} is a weakly increasing sequence $1\leq \lambda_1\leq \lambda_2 \leq \cdots\leq \lambda_l$
such that $\lambda_1+\cdots+\lambda_l=n$. The number $l$ of integers in the sequence is called the {length of the partition}, and the value $\lambda_l$ is the {largest term}. The set of partitions with length $l$ is denoted by $\Lambda_l$, and the set of partitions with length $l$ and largest term at
most $m$ is denoted by $\Lambda_{l,m}$. For every $i$ with $1\leq i\leq m$, if $\lambda\in \Lambda_{l,m}$, we define $\overline{\lambda}\in \Lambda_{l,i}$ by putting $\overline{\lambda}_k=\min\{\lambda_k,i\}$ for every $k$ with $1\leq k\leq l$. We thus obtain a natural map $\Lambda_{l,m}\rightarrow \Lambda_{l,i}$.

Fix an effective divisor $E$ of degree $e\geq 2g-d-1$ on $C$. We now associate a partition to every $\cL_m\in \pic^d(C)_m$. ${p}_*{(\cL_m\otimes q^*(\cE))}$ and ${p}_*(\cL_m \otimes q^*(\cE)\otimes\cO_{q^{-1}E})$ are locally free sheaves on $T_m$, hence they are finitely generated free modules over $k[t]/(t^{m+1})$.

\begin{definition}{\label{lambda}}
A family of line bundles $\cL_m$ of degree $d$ on $C$ over $T_m$ is called of type $\lambda\in \Lambda_{l,m+1}$ if there are bases of ${p}_*
{(\cL_m\otimes q^*(\cE))}$ and ${p}_*(\cL_m \otimes q^*(\cE)\otimes\cO_{q^{-1}E})$ in which $\Phi_{\cL_m}$ is represented by the matrix in $M_{(d+e+1-g)\times e}(k[t]/(t^{m+1}))$
~~~
\vspace{1cm}

 \[ \left( \begin{array}{cccccccc}
1& &  & &  & &0 & 0\\
&\ddots && && &&\\
& & 1& &  && & \\
& & &t^{\lambda_1} & & &\vdots&\vdots\\
&&&& \ddots &&&\\
& && && t^{\lambda_l}&0 &0\\
 \end{array} \right)\]
 \begin{picture}(0,0)
\put(225,75){\Large $0$}
\put(185,22){\Large $0$}
\end{picture}
\end{definition}

\begin{definition}\label{nrlambda}
Given a partition $\lambda$ , let  $r_i(\lambda)$ be the number of $k$ such that $\lambda_k=i$  and let $n_i(\lambda)$ be the number of $k$
such that $\lambda_k\geq i$.
\end{definition}

It is easy to see that the partition $\lambda$ in Definition \ref{lambda} does not depend on the choice of bases. If $L$ is the image of $\cL_m$ under the truncation map $\pi_m: \pic^d(C)_m\rightarrow \pic^d(C)$, then we will see below that the length of the partition associated to $\cL_m$ is $h^0(C,L)$.

We now give a criterion to decide whether an element $\cL_m\in \pic^{g-1}(C)_m$ is a jet of $\Theta$ in terms of the partition $\lambda$.
\begin{lemma}\label{condition}
For every family of line bundles $\cL_m\in \pic^{g-1}(C)_m$  centered at $L\in \Pic^{g-1}(C)$ and of type $\lambda\in \Lambda_{l,m+1}$, the following are equivalent:
\begin{enumerate}
\item[(i)] $\cL_m\in \Theta_{m,L}$.

\item[(i){\scriptsize${}^\prime$}] $\det(\Phi_{\cL_m})=0$ in $k[t]/(t^{m+1})$.

\item[(ii)] $\sum\limits_{i=1}^l{\lambda_i}\geq m+1$.

\item[(ii){\scriptsize${}^\prime$}] $\sum\limits_{j=1}^{m+1} r_j(\lambda)\cdot j\geq m+1$.

\item[(ii){\scriptsize${}^{\prime\prime}$}] $\sum\limits_{k=1}^{m+1} n_k(\lambda)\geq m+1$.
\end{enumerate}
\end{lemma}

\begin{proof}
Recall that $\Theta = W^{0}_{g-1}\subset \pic^{g-1}(C)$.  With the above notation, for every family of line bundles $\cL_m$ in $\Theta_m$, the sheaves ${p}_*{(\cL_m\otimes q^*(\cE))}$ and ${p}_*(\cL_m \otimes q^*(\cE)\otimes\cO_{q^{-1}E})$ are locally free of rank $e$.  The definition shows that the theta divisor parameterizes the line bundles $\cL_m$ for which  $\det(\Phi_{\cL_m})=0$. This proves the equivalence between (i) and (i){\scriptsize${}^\prime$}.
It is clear that with the choice of basis in Definition \ref{lambda}, $\det(\Phi_{\cL_m})=t^{\lambda_1+\cdots +\lambda_l}\in k[t]/(t^{m+1})$. Therefore the determinant vanishes if and only if $\sum\limits_i^l{\lambda_i}\geq m+1$.

In order to complete the proof of the lemma, it suffices to show that $$\sum\limits_{i=1}^l{\lambda_i}=\sum\limits_{j=1}^{m+1} r_j(\lambda)\cdot j=\sum\limits_{k=1}^{m+1} n_k(\lambda).$$ The first equality is
clear by the definition of $r_j(\lambda)$. The second equality follows from $n_k(\lambda)=\sum\limits_{j\geq k} r_j(\lambda)$. Indeed, $\sum\limits_{k=1}^{m+1}
n_k(\lambda)=\sum\limits_{k=1}^{m+1} \sum\limits_{j\geq k} r_j(\lambda)= \sum\limits_{j=1}^{m+1} r_j(\lambda)\cdot j$.
\end{proof}

Using the definition of $W^r_d(C)$, we have the following description of $W^r_d(C)_{m,L}$, which gives a generalization of Lemma \ref{condition}.

\begin{lemma}\label{conditionw} Let $\cL_m\in \pic^d(C)_m$ have type $\lambda=(1\leq \lambda_1\leq \cdots\leq\lambda_l\leq m+1)$. The following are equivalent:
\begin{enumerate}
\item[(i)] $\cL_m\in W^r_d(C)_{m,L}$.

\item[(i')] All the $(e+d+1-g-r)$ minors of $\Phi_{\cL_m}$ vanish in $k[t]/(t^{m+1})$.

\item[(ii)] $\sum\limits_{i=1}^{l-r }{\lambda_i}\geq m+1$.

\item[(ii')] $\sum\limits_{i=1}^{l-r}(l-i-r+1)(\lambda_i-\lambda_{i-1})\geq m+1$, where $\lambda_0=0$.
\end{enumerate}
\end{lemma}

The proof of this lemma is very similar to that of Lemma \ref{condition}, so we leave it to the reader.

Our first goal is to recover Riemann's Singularity Theorem using jet schemes.

\begin{theorem}\label{multiplicity thm}
For every $L\in \Theta$, we have ${\rm{mult}}_{L}\Theta=h^0(C,L)$.
\end{theorem}

\begin{remark}\label{smoothness}
Note that the multiplicity of a divisor at a point is one if and only if the divisor is smooth at that point, hence Theorem \ref{multiplicity thm} implies in particular that a line bundle $L\in \Theta$ is a smooth point if and only if $h^0(C,L)=1$.
\end{remark}

Before proving the theorem we need some preparations. For every degree $d$ line bundle $L$, we shall first describe the fiber of $\rho^m_{m-1}: \pic^{d}(C)_{m,L}\rightarrow \pic^d(C)_{m-1,L}$.

Let $E$ be the effective divisor  of degree $e\geq 2g-d-1$ in Definition \ref{lambda}. By the universal property of $\pic^d(C)$, every $\cL_m \in \pic^d(C)_{m,L}$ is identified with a line bundle on $C\times T_m$.  Let us fix a line bundle $L\in \pic^d(C)$ and a family of line bundles $\cL_m\in \pic^d(C)_{m,L}$ lying over $L$. For every $0\leq i\leq m$, we denote by $\cL_i$ the image of $
\cL_m$ in $\pic^d(C)_{i,L}$ under the truncation map $\pic^d(C)_m\rightarrow \pic^d(C)_i$.  By the short exact sequence \eqref{ses}, $H^0(\cL_i)$ is the kernel of the morphism
\begin{equation*}
\Phi_{\cL_i}: M_i=H^0(\cL_i\otimes q^*(\cE))\rightarrow N_i=H^0(\cL_i\otimes q^*(\cE)\otimes\cO_{q^{-1}E}).
\end{equation*}

There is a $k[t]/(t^{m+1})$-module map $\pi_{i}^m: H^0(\cL_m)\rightarrow H^0(\cL_i)$ induced by restriction of sections. This can be described as follows. Applying the Base-change Theorem to the morphism $T_{i}\hookrightarrow T_{m}$, we
obtain the following commutative diagram
\begin{equation*}
\begin{array}[c]{ccccc}
H^0(\cL_m)&{\hookrightarrow}& M_m& \xrightarrow{\Phi_{\cL_m}} &N_m\\
\downarrow\scriptstyle{\pi_i^m}&&\downarrow\scriptstyle{\rho_M}&&\downarrow\scriptstyle{\rho_N}\\
H^0(\cL_{i})&{\hookrightarrow}&M_{i}&{\xrightarrow{\Phi_{\cL_i}}} &N_i
\end{array}
\end{equation*}
Clearly $M_i=M_m \otimes_{k[t]/(t^{m+1})} k[t]/(t^{i+1})$ and $N_i=N_m \otimes_{k[t]/(t^{m+1})} k[t]/(t^{i+1})$ and the vertical maps are induced by the quotient map $ k[t]/(t^{m+1})\rightarrow k[t]/(t^{i+1})$.

\begin{lemma}\label{filtration lemma}
For every $0\leq i\leq m$, there is an embedding of $k[t]/(t^{m+1})$-modules
$$v^m_i: H^0(\cL_i)\hookrightarrow H^0(\cL_m)$$
such that the image is the kernel of $\pi_{m-i-1}^m: H^0(\cL_m)\rightarrow H^0(\cL_{m-i-1})$.
\end{lemma}

\begin{proof}

The multiplication with $t^{m-i}$ defines a linear map of $k[t]/(t^{m+1})$-modules $$k[t]/(t^{i+1})\rightarrow k[t]/(t^{m+1})$$ and induces embeddings of $k[t]/(t^{m+1})$ modules $M_i
\xrightarrow{u^m_i} M_m$ and $N_i\xrightarrow{w^m_i} N_m$. Therefore it induces an injective $k[t]/(t^{m+1})$-module morphism $v^m_i: H^0(\cL_i)\rightarrow
H^0(\cL_m).$

It is clear that the image of the embedding $u_i^m: M_i \rightarrow M_m$ is $\Ann_{M_m}(t^{i+1})$. By definition, we have $H^0(\cL_m)
\cap \Ann_{M_m}(t^{i+1})=\Ann_{H^0(\cL_m)}(t^{i+1})$. The multiplication map $w^m_i: N_i\rightarrow N_m$ is injective, and one deduces easily that the image of $v_i$ is $\Ann_{H^0(\cL_m)}(t^{i+1})$. Since $\ker\pi_{m-i-1}^m=\Ann_{H^0(\cL_m)}(t^{i+1})$, this completes our proof.
\end{proof}

\begin{lemma}\label{dimension lemma}
For every family of line bundles $\cL_m\in \pic^d(C)_m$ of type $\lambda\in \Lambda_{l,m+1}$, we have
$$h^0(\cL_m)=\sum\limits_{k=1}^{m+1} n_k(\lambda).$$
\end{lemma}

\begin{proof}
Choose bases $\{e_j\}$ and $\{f_h\}$ for the free modules $M_m$ and $N_m$ such that $\Phi_{\cL_m}$ is represented by the matrix
\begin{equation*}
\nonumber \left( \begin{array}{cccccccc}
1& &  & &  & &0 & 0\\
&\ddots && && &&\\
& & 1& &  && & \\
& & &t^{\lambda_1} & & &\vdots&\vdots\\
&&&& \ddots &&&\\
& && && t^{\lambda_l}&0 &0\\
 \end{array} \right)
 = A_0+A_1\cdot t +\cdots+ A_m\cdot t^m.
\end{equation*}
\begin{picture}(0,0)
\put(148,92){\Large $0$}
\put(100,30){\Large $0$}
\end{picture}

All  $A_i$ are $(d+e+1-g)\times e$ matrices over the field $k$.
For every $0\leq i\leq m$ the image of $\{e_j\}$ under the map $\rho_M: M_m\rightarrow M_i$ gives a basis of $M_i$ over $k[t]/(t^{i+1})$. Similarly, the image of $\{f_h\}$ under $\rho_N: N_m\rightarrow N_i$ gives a basis of $N_i$. With respect to these bases, the homomorphism $\Phi_{\cL_i}$ is represented by the matrix $A_0+A_1\cdot t +\cdots+ A_i\cdot t^i$.

We first consider the case $m=0$. $\Phi_L$ is represented by $A_0$, which is a diagonal matrix with $1$ showing up on the first $e+d+1-g-l$ rows,
hence $h^0(L)=\dim_k\ker\Phi_L=l=n_1(\lambda)$.

Let $\lambda'$ be the type of $\cL_{m-1}$. One can check easily that $\lambda'$ is the image of $\lambda$ under the natural map $\Lambda_{l,m+1}\rightarrow \Lambda_{l,m}$. For $k\leq m$, we have $n_k(\lambda')=n_k(\lambda)$. Now it suffices to show that $h^0(\cL_m)-h^0(\cL_{m-1})=n_{m+1}(\lambda)$ for $m\geq 1$. For each $i>0$, $A_i$ is a diagonal matrix with entries $0$ or $1$, with $1$'s  in the rows $(e+d+1-g)-l+r_1+\cdots+r_{i-1}+j$,  with $1\leq j\leq r_{i}$, where $r_i=r_i(\lambda)$ (See Definition \ref{nrlambda}). We now consider $\{t^k\cdot e_j\}$ and $\{t^k\cdot f_h\}$ where $0\leq k\leq m$ to be the bases of $M_m$ and $N_m$,  respectively, as linear spaces over $k$. The matrix associated to $\Phi_{\cL_m}$ as a morphism of $k-$linear spaces has the upper triangular form

\begin{equation*}
\nonumber  \Psi_{\cL_m}=\left( \begin{array}{cccccccc}
A_0&A_1 &A_2 &\cdots &A_{m-1}& A_m&0&0\\
& A_0&A_1  &\cdots& A_{m-2}&A_{m-1}&0&0\\
&&A_0 &\ddots  &  &\vdots&\vdots &\vdots \\
&&&\ddots&\ddots &\vdots&\vdots&\vdots\\
&&&  & A_0&A_1&0&0\\
&& && &A_0&0&0\\
 \end{array} \right)
 \end{equation*}
 \begin{picture}(0,0)
\put(170,30){\Large $0$}
\end{picture}

Therefore the associated matrix $\Psi_{\cL_{m-1}}$ of $\Phi_{\cL_{m-1}}$ as a $k$-linear map is the bottom right corner submatrix of the associated matrix of $\Phi_{\cL_m}$, obtained by omitting the rows and columns containing the left upper corner $A_{0}$.

In each row and column of the matrix $\Psi_{\cL_m}$, there is at most one nonzero element. Therefore $\rank \Psi_{\cL_m}=\rank \Psi_{\cL_{m-1}}+ \sum\limits_{i=0}^{m} \rank A_i$.  Since $\rank(A_0)=d+e+1-g-l$ and $\rank(A_i)=r_i(\lambda)$ for $1\leq i\leq m$, we deduce that $\dim_k\ker \Phi_{\cL_m}-\dim_k\ker\Phi_{\cL_{m-1}}=n_{m+1}(\lambda)$.  Therefore $h^0(\cL_m)-h^0(\cL_{m-1})=n_{m+1}(\lambda)$.
\end{proof}

\begin{remark}\label{image}
For every $j$ with $0\leq j\leq m$, Lemmas \ref{filtration lemma} and \ref{dimension lemma} imply that the image of the morphism $\pi^j_0: H^0(\cL_j)\rightarrow H^0(L)$ has dimension equal to
$$h^0(\cL_{j})-\dim_k \ker(\pi^j_0)=h^0(\cL_{j})-h^0(\cL_{j-1})=n_{j+1}(\lambda).$$ Therefore $\pi^{j}_0$ is a zero map if and only if $n_{j+1}(\lambda)=0$.
\end{remark}

We now fix a line bundle $L\in \pic^d(C)$ and describe the fibers of the truncation maps $\rho^{m}_{m-1}: \pic^d(C)_{m,L}\rightarrow \pic^d(C)_{m-1,L}$ for every $m$. Let $\{U_\alpha\}$ be an affine covering of $C$ which trivializes the line bundle $L$ by isomorphisms $\gamma_\alpha: L|_{U_\alpha}\cong \cO_{U_\alpha}$. Let $\{g_\alb=\gamma_\beta\circ \gamma_\alpha^{-1}\}$ be the corresponding transition functions. For every scheme $U_{\alpha}$ and $i\geq 1$, we have short exact sequence of sheaves on $U_{\alpha}\times T_i$ as follows:
\begin{eqnarray*}
0\rightarrow \mathcal{O}_{U_\alpha}\rightarrow \mathcal{O}_{U_\alpha\times T_i}^*\rightarrow \cO_{U_{\alpha}\times T_{i-1}}^*\rightarrow 0,
\end{eqnarray*}
where the embedding morphism maps $x\in \cO_{U_{\alpha}}$ to $1+x\cdot t^i$.
Since $U_{\alpha}$ is affine, $H^j(\cO_{U_{\alpha}})$ vanishes for every $j\geq 1$. We thus obtain an isomorphism $H^{1}(\cO_{U_{\alpha}\times T_i})\cong H^{1}(\cO_{U_{\alpha}\times T_{i-1}})$. In other words, we have $\pic(U_{\alpha}\times T_{i})\cong \pic(U_{\alpha}\times T_{i-1})$. By induction on $i$ with $0\leq i\leq m$, we deduce that $\{U_{\alpha}\times T_m\}$ is an affine covering of $C\times T_m$ which trivializes every line bundle $\cL_m\in \pic^d(C)_{m,L}$.
In particular, for every line bundle $\cL_1\in \pic^d(C)_{1,L}$ on $C\times T_1$, there is a trivialization for $\cL_1$ on the covering $\{U_\alpha\times T_1\}$ with the transition functions $\{g_\alb(1+t\phi^{(1)}_\alb)\}$.
This gives a bijection $\xi: \pic^d(C)_{1,L}\rightarrow  H^1(C,\cO_C)$ via $\xi(\cL_1)=[\phi^{(1)}_\alb]$.

In general, we fix a family of line bundles $\cL_{m-1}\in \pic^d(C)_{m-1,L}$. After we also fix a point $\mathcal{M}$ in the fiber of $\rho^{m}_{m-1}$ over $\cL_{m-1}$, we get an isomorphism $$(\rho^m_{m-1})^{-1}(\cL_{m-1})\cong H^1(C,\cO_C).$$  Since we will use later the description in terms of  \u{C}ech cohomology classes, we describe this isomorphism as follows. We choose a trivialization  of $\cL_{m-1}$ with the transition functions $g_\alb^{m-1}:=g_\alb(1+t\phi^{(1)}_\alb+\cdots +t^{m-1}\phi^
{(m-1)}_\alb)$. It is easy to see that there is a trivialization for $\mathcal{M}$ with transition functions $g_\alb^m=:g_\alb(1+t
\phi^{(1)}_\alb+\cdots +t^{m-1}\phi^{(m-1)}_\alb+t^{m}\phi^{(m)}_\alb)$.

Every point $\cL_m\in (\rho^m_{m-1})^{-1}(\cL_{m-1})$ has transition functions
$$g_\alb(1+t\phi^{(1)}_\alb+\cdots +t^{m-1}\phi^
{(m-1)}_\alb+t^{m}(\phi^{(m)}_\alb+\psi_{\alb}))$$
where $[\psi_{\alb}]\in H^1(C,\cO_C)$. We thus obtain an isomorphism
$$\xi: (\rho^m_{m-1})^{-1}(\cL_{m-1})\rightarrow H^1(C,\cO_C)$$ given by $\xi(\cL_m)=[\psi_{\alb}]$. Abusing the notation, we write $[\cL_m]$ for the cohomology class corresponding to $\cL_m$. Note, however, that this depends on the choice of $\mathcal{M}$.

Let $s_{m-1} \in H^0(\cL_{m-1})$ be a nonzero section. The obstruction to extending $s_{m-1}$ to a section of $\cL_m$ can be described as follows. We have a short exact sequence of sheaves on $C\times T_m$,
$$0\rightarrow L\rightarrow \cL_m\rightarrow \cL_{m-1}\rightarrow 0.$$
Let $\delta_{\cL_{m}}$ be the connecting map $H^0(\cL_{m-1})\rightarrow H^1(C,L)$.  The long exact sequence on cohomology implies that $s_{m-1}$ can be extended to a section $s_m$ of $\cL_{m}\in (\rho^m_{m-1})^{-1}(\cL_{m-1})$ if and only if $\delta_{\cL_{m}}(s_{m-1})=0$.

With the above notation, we get the following more explicit obstruction to extending a section of $\cL_{m-1}$ in terms of  \u{C}ech cohomology.

\begin{lemma}\label{equationoftransitionfunction}
Fix a line bundle $\mathcal{M}$ in the fiber of $\rho^m_{m-1}$ over $\cL_{m-1}$. For a fixed section $s_{m-1}=(\sum\limits_{j=0}^{m-1}c_{\alpha}^{(j)}t^j)\in H^0(\cL_{m-1})$, let $s_0$ be the its image under $\pi^{m}_{0}: H^0(\cL_{m-1})\rightarrow H^0(L)$.
The section $s_{m-1}$ has an extension to a section of $\cL_m$ if and only if
\begin{equation}\label{commonsol}\tag{$\dagger$}
\nu(s_0\otimes [\cL_m]) {\rm ~is ~the~ cohomology ~class ~corresponding ~to~ } (- \gamma_\alpha^{-1}(\sum\limits_{j=1}^{m}\phi_\alb^{(j)} c^{(m-j)}_\alpha))
\end{equation}
where $\nu$ is the natural pairing $H^0(C,L)\otimes H^1(C,\cO) \rightarrow H^1(C,L)$.
\end{lemma}

\begin{proof} Assume there is an extension of $s_{m-1}\in H^0(\cL_{m-1})$ to a section $s_m\in H^0(\cL_m)$.
Locally $s_{m-1}$ is given by functions $\sum\limits_{j=0}^{m-1}c^{(j)}_\alpha t^j\in \Gamma(U_\alpha\times T_{m-1}, \cO_{U_\alpha\times T_{m-1}})$. We write $s_m$ as $\sum\limits_{j=0}^{m-1}c^{(j)}_\alpha t^j+c^{(m)}_\alpha t^m$. Let $\gamma^{m}_{\alpha}: \cL_{m}|{_{U_\alpha\times T_m}}\rightarrow \cO_{U_\alpha\times T_m}$ be a trivialization of $\cL_m$ on $U_{\alpha}\times T_m$. We thus have the following equality on $(U_\alpha\cap U_\beta)\times T_m$:
$$(\gamma^{m}_{\alpha})^{-1}(\sum\limits_{j=0}^{m}c^{(j)}_\alpha t^j)=(\gamma^{m}_{\beta})^{-1}(\sum\limits_{j=0}^{m}c^{(j)}_\beta t^j).$$
Since $(\gamma^{m}_{\alpha})^{-1}=(\gamma^{m}_{\beta})^{-1}\circ g^{m}_\alb$, we have
$(\gamma^{m}_{\beta})^{-1}
 \circ g^{m}_\alb(\sum\limits_{j=0}^{m}c^{(j)}_\alpha t^j)=(\gamma^{m}_{\beta})^{-1}(\sum\limits_{j=0}^{m}c^{(j)}_\beta t^j).$
More explicitly, we obtain
$$g_\alb(1+t
\phi^{(1)}_\alb+\cdots+t^{m-1}\phi^{(m-1)}_\alb+t^{m}(\phi^{(m)}_\alb+\psi^{(m)}_\alb))(\sum\limits_{j=0}^{m}c^{(j)}_\alpha t^j)=(\sum\limits_{j=0}^{m}c^{(j)}_\beta t^j)$$
in $\cO((U_{\alpha}\cap U_{\beta})\times T_m)$.
We now expend this equation and take the coefficient of $t^i$ for $i$ with $0\leq i\leq m$. If $i<m$, the equation we obtain from the coefficient of $t^i$ always holds since $s_{m-1}$ is a section of $\cL_{m-1}$. For $i=m$, we obtain
$$g_\alb(\psi_\alb\cdot c^{(0)}_\alpha+\sum\limits_{j=1}^m \phi_\alb^{(j)}\cdot c_\alpha^{(m-j)}+c_\alpha^{(m)})=(c_\beta^{(m)})$$
in $\mathcal{O}((U_{\alpha}\cap U_\beta)\times T)$. Note that the restriction to the  trivialization $\gamma^m_\alpha$ to the subsheaf $L$ of $\cL_m$ is exactly the trivialization $\gamma_\alpha$, we have $$(\gamma_{\beta})^{-1}\circ g_\alb(\psi_\alb\cdot c^{(0)}_\alpha+\sum\limits_{j=1}^m \phi_\alb^{(j)}\cdot c_\alpha^{(m-j)}+c_\alpha^{(m)})= (\gamma_{\beta})^{-1}(c_\beta^{(m)})$$
as sections of $L$ on $(U_{\alpha}\cap U_\beta)$.

Clearly $(\gamma_{\beta})^{-1}\circ g_\alb (c_\alpha^{(m)})-(\gamma_{\beta})^{-1}(c_\beta^{(m)})$ gives the zero cohomology class in $H^1(C,L)$. We obtain that $\nu(s_0\otimes [\cL_m])$, the cohomology class corresponding to $(\gamma_{\alpha}^{-1}(\psi_\alb\cdot c^{(0)}_\alpha))$ is equal to the cohomology class corresponding to $(- \gamma_{\alpha}^{-1}(\sum\limits_{j=1}^{m}\phi_\alb^{(j)} c^{(m-j)}_\alpha))$. By reversing the argument, we also obtain the converse.
\end{proof}

\begin{remark}
The identification between the fiber of $\pic^d(C)_{m,L}\rightarrow \pic^d(C)_
{m-1,L}$ and $H^1(C,\cO_C)$ is not canonical. In particular, the expression for $\gamma(s_0\otimes[\cL_m])$ in (\ref{commonsol}) does depend on $\mathcal{M}$.  However, for any fixed nonzero section $s_{m-1}$, the dimension of the subset $$\{\cL_m\in (\rho^m_{m-1})^{-1}(\cL_{m-1})~|~ H^0(\cL_m)\rightarrow H^0(\cL_{m-1}) {\rm ~has ~nonempty~fiber ~over~ } s_{m-1}\} $$
is independent of $\mathcal{M}$.
\end{remark}

We now prove Theorem \ref{multiplicity thm}. The idea is similar to that in Kempf's proof of Riemann's multiplicity formula.

\noindent{\it Proof} of Theorem \ref{multiplicity thm}.
For every effective Cartier divisor $D$ on a smooth variety $X$ and a point $x\in D$, the multiplicity of $D$ at $x$ is equal to the minimal positive integer $m$ such that $D_{m,x}$ is a proper subset of $X_{m,x}$.

Let $L\in \Theta$ be a line bundle with $l=h^0(L)$. We first show that $\Theta_{m,L}=\pic^{g-1}(C)_{m,L}$ for every $m<l$. This follows from the description of $\Theta$ as a determinantal variety. Indeed, let $\cL_m\in \pic^{g-1}(C)_{m,L}$ be a line bundle of type $\lambda\in \Lambda_{l,m}$, then $\sum\limits_{i=1}^{l}\lambda_i\geq l>m$. By Lemma \ref{condition}, we have $\cL_m\in \Theta_{m,L}$. Hence $\Theta_{m,L}=\pic^{g-1}(C)_{m,L}$ for every $m<l$.

We now show that $\Theta_{m,L}\neq \pic^{g-1}(C)_{m,L}$ for $m=l$. Let $\mathcal{Z}_1$ be the image of $\Theta_{m,L}$ under $\pic^{g-1}(C)_m\rightarrow \pic^{g-1}(C)_1$. It suffices to show that $\mathcal{Z}_1\neq \pic^{g-1}(C)_{1,L}$. For every $\cL_m\in \Theta_{m,L}$ of type $\lambda=(1\leq\lambda_1\leq\cdots\leq\lambda_l)$, Lemma \ref{condition} implies that $\sum\limits_{i=1}^{l}\lambda_i \geq m+1$.  Hence $\lambda_l\geq 2$ and $n_2(\lambda)\geq 1$. By Lemma \ref{dimension lemma}, we have $h^0(\cL_1)-h^0(L)=n_2(\lambda)\geq 1$. By Remark \ref{image}, we see that the map $\pi^1_0: H^0(\cL_1)\rightarrow H^0(L)$ is not zero. Equivalently, there is a nonzero section $s_0 \in H^0(L)$ which can be extended to a section of $H^0(\cL_1)$. Let $\mathcal{Z}_2$ be the subset
$$\{\cL_1 \in \pic^{g-1}(C)_1~|~\pi^1_0: H^0(\cL_1)\rightarrow H^0(L) {\rm ~is ~not ~zero}\}.$$
 We have seen that $\mathcal{Z}_1$ is a subset of $\mathcal{Z}_2$, hence it is enough to show that $\mathcal{Z}_2\neq \pic^{g-1}(C)_{1,L}$.

We now apply Lemma \ref{equationoftransitionfunction} with $m=1$. Let $\mathcal{M}$ be the trivial deformation of $L$, i.e. $\mathcal{M}$ represents the zero tangent vector at $L$. To compute the dimension of $\mathcal{Z}_2$, we consider the proper subset
$$\mathcal{Z}=\{(W, \cL_1) ~|~\nu(s_0\otimes [\cL_1])=0 {\rm~ for~ every~} s_0\in W \}$$
of $\PP (H^0(C,L))\times H^1(C,\cO_C)$. Here $\PP (H^0(C,L))$ stands for the projective space of one dimensional subspaces of $H^0(C,L)$. Let $W$ be an element in $\PP (H^0(C,L))$ and $s_0$ a nonzero element of $W$. The induced map $H^1(C,\cO_C)\rightarrow H^1(C,L)$ taking $[\cL_1]$ to $\nu(s_0\otimes [\cL_1])$ is surjective. Hence each fiber of the first projection map $\mathcal Z \rightarrow \PP (H^0(C,L))$ is a codimension $l$ vector space of $H^1(C,\cO_C)$. We obtain $\dim\mathcal Z=g-1$. Since $\mathcal{Z}_2$ is a subset of the image of the second projection map $\mathcal{Z}\rightarrow H^1(C,\cO_C)$, we obtain $\dim\mathcal{Z}_2\leq g-1$. Hence $\mathcal{Z}_2\neq \pic^{g-1}(C)_{1,L}$. This completes the proof.
\hfill{$\Box$}

For smooth projective curves of genus $g\leq 2$, Riemann's Singularity Theorem implies that the theta divisor is smooth. We consider the singularities of the theta divisor for curves of genus $g\geq 3$ in the next section.

\section{Singularities of the Theta divisor and of the $W^r_d$ loci}

Our first goal in this section is to give an upper bound for $\dim W^r_d(C)_{m,L}$ for each $L\in \pic^d(C)$ and $m\geq 0$. We fix a line bundle $L$ of degree $d$ with $l=h^0(L)$.

For every partition $\lambda\in\Lambda_{l,m+1}$, we denote by $C_{\lambda,m}$ the subset $$\{\cL_m\in \pic^d(C)_{m,L}~|~\cL_m \text{ is of type } \lambda\}.$$ It is easy to see that
 locally $C_{\lambda,m}$ is the pull back of a locally closed subset of the $m$-th jet scheme of the variety of $(d+e+1-g)\times e$ matrices. Therefore $C_{\lambda,m}$ is a constructible subset of $\pic^d(C)_{m,L}$. By Lemma \ref{conditionw}, we have $W^r_d(C)_{m,L}=
\bigcup\limits_{\lambda} C_{\lambda,m}$, where $\lambda$ varies over the partitions in  $\Lambda_{l,m+1}$ satisfying  $\sum\limits_{i=1}^{l-r }{\lambda_i}\geq m+1$. In particular, we have a finite union $\Theta_{m,L}=\bigcup\limits_{\lambda}C_{\lambda,m}$, where $\lambda$ varies over all elements in $\Lambda_{l,m+1}$ with $\sum\limits_{i=1}^{l}{\lambda_i}\geq m+1$. In order to estimate the dimension of $\Theta_{m,L}$, it is enough to bound the dimension of $C_{\lambda,m}$ for every $\lambda\in \Lambda_{l,m+1}$. The idea is to describe the image of $C_{\lambda,m}$ under the truncation map $\rho^m_i: \pic^d(C)_{m,L}\rightarrow \pic^d(C)_{i,L}$ for every $i\leq m$.

\begin{definition}
{\it A weak flag of} $H^0(C,L)$ {\it of signature} $\kappa=(\kappa_i)$ with $\kappa_1\geq \cdots \geq\kappa_n $ is a sequence of subspaces of $H^0(C,L)$,
$$\textbf{V}: ~H^0(C,L)=V_0\supseteq V_1\supseteq\cdots \supseteq V_{n-1}\supseteq V_n$$
such that $\dim V_i=\kappa_i$ for every $1\leq i\leq n$.  Here $n$ is called the {\it length} of the weak flag $\textbf{V}$.
Given a weak flag $\textbf{V}$ of $H^0(C,L)$ of length $n$, for every $i\leq n$ we denote by $\textbf{V}_{(i)}$ the truncated weak flag of length $i$:
$$\textbf{V}_{(i)}: H^0(C,L)=V_0\supseteq V_1\supseteq\cdots \supseteq V_{i-1}\supseteq V_i.$$
\end{definition}

For every $\cL_m\in C_{\lambda,m}$ and every $j$ with $0\leq j\leq m$, we denote by $\cL_j$ the image of $\cL_m$ under $\rho^m_j: \pic^d(C)_m\rightarrow \pic^d(C)_j$. Lemma \ref{dimension lemma} implies that the function $C_{\lambda,m}\rightarrow \ZZ$ which takes $\cL_m$ to $h^0(\cL_j)=\sum\limits_{k=1}^{j+1} n_k(\lambda)$ is constant. For a fixed $\cL_m\in C_{\lambda,m}$, the images $V_j$ of the morphisms $\pi^j_0: H^0(\cL_j)\rightarrow H^0(L)$ give a weak flag $\textbf{V}_{\cL_m}$ of $H^0(L)$ of length $m$.  Remark \ref{image} implies that $\dim V_j=\dim H^0(\cL_{j})-\dim H^0(\cL_{j-1})=n_{j+1}(\lambda)$. Hence the signature $\kappa$ of the weak flag $\textbf{V}_{\cL_m}$, with $\kappa_j=n_{j+1}(\lambda)$, only depends on the partition $\lambda$.

Lemma \ref{filtration lemma} shows that there is a short exact sequence $$0\rightarrow H^0(\cL_{m-1})\xrightarrow{v^m_{m-1}} H^0(\cL_{m})\twoheadrightarrow V_m\rightarrow 0.$$
We now choose a splitting of this short exact sequence, which gives a decomposition
$H^0(\cL_{m})=H^0(\cL_{m-1})\oplus \widetilde{V}_m\,,$ with $\widetilde{V}_m$ mapping isomorphically onto $V_m$. The restriction map
$\pi^m_{m-1}: H^0(\cL_{m})\rightarrow H^0(\cL_{m-1})$ maps $\widetilde{V}_m$ isomorphic to its image. For the short exact sequence $$0\rightarrow H^0(\cL_{m-2})\xrightarrow{v^{m-1}_{m-2}} H^0(\cL_{m-1})\twoheadrightarrow V_{m-1}\rightarrow 0,$$
we can choose a splitting $H^0(\cL_{m-1})=H^0(\cL_{m-2})\oplus \widetilde{V}_{m-1}$ such that the restriction map $\pi^m_{m-1}$ maps $\widetilde{V}_m$ into $\widetilde{V}_{m-1}$. By descending induction on $i$ with $0\leq i\leq m$, we can find a subspace $\widetilde{V_i}\subset H^0(\cL_i)$ for each $i$
such that
\begin{enumerate}
  \item The restriction of the truncation map $\pi^i_0: H^0(\cL_{i})\rightarrow H^0(L)$ to $\widetilde{V}_i$ induces an isomorphism onto $V_i$.
  \item The truncation map $\pi^i_{i-1}: H^0(\cL_{i})\rightarrow H^0(\cL_{i-1})$ takes $\widetilde{V}_i$ into $\widetilde{V}_{i-1}$
\end{enumerate}

\begin{definition}
A weak flag $\textbf{V}$ of $H^0(C,L)$ of length $m$ is {\it extended compatibly} to the line bundle $\cL_m$ if there
are linear subspaces $\widetilde{V}_i\subset H^0(C,\cL_i)$ for each $i\leq m$ such that (1) and (2) above hold.
\end{definition}
In this case, the set of linear subspaces $\{\widetilde{V}_i\}_i$ as above is called a {\it compatible extension} of $\textbf{V}$ to line bundle $\cL_m$.
The above argument shows that every weak flag $\textbf{V}_{\cL_m}$ associated to a line bundle $\cL_m$ can be extended compatibly to the line bundle $\cL_m$.

 For every $i$ with $1\leq i\leq m$, recall that $\overline{\lambda}$ is the image of $\lambda\in \Lambda_{l,m+1}$ under the map $\Lambda_{l,m+1}\rightarrow \Lambda_{l,i+1}$.  Given a weak flag $\textbf{V}$ of $H^0(L)$ of length $m$, we denote by $S_{i,\textbf{V}}^{\lambda}$ the set of line bundles $\cL_i\in \pic^d(C)_{i,L}$ such that $\cL_i\in C_{i,\overline{\lambda}}$ and $\textbf{V}_{(i)}$ can be extended compatibly to $\cL_i$. For a fixed non-increasing sequence $\kappa$, we define $S_{i,\kappa}^{\lambda}=\bigcup\limits_{\textbf{V}'}S_{i,\textbf{V}'}^{\lambda}$, where $\textbf{V}'$ varies over all weak flags of $H^0(L)$ of signature $\kappa$. For convenience, we set $S_{0,\textbf{V}}^{\lambda}=S_{0,\kappa}^{\lambda}=\{L\}$.

Standard arguments show that $S_{i,\textbf{V}}^{\lambda}$ and $S_{i,\kappa}^{\lambda}$ are constructible subsets of $\pic^d(C)_{i,L}$. For the benefit of the reader, we give the details in the appendix. The truncation map $\rho^m_i: \pic^d(C)_{m,L}\rightarrow \pic^d(C)_{i,L}$ maps $C_{\lambda,m}$ to the set $S_{i,\kappa}^{\lambda}$ with $\kappa_j=n_{j+1}(\lambda)$. In order to estimate the dimension of $C_{\lambda,m}$, we only need to estimate $\dim S_{i,\kappa}^{\lambda}$ for a suitable $i\leq m$.

\begin{definition}
For a fixed weak flag $\textbf{V}$ of $H^0(L)$ of length $m$, for every $i$ and $j$ with $1\leq i\leq j\leq m$,  we define $\widetilde{S}_{i,j,\textbf{V}}^{\lambda}$ to be the set of pairs $(\cL_i, W)$ such that
\begin{enumerate}
  \item $\cL_i\in S_{i,\textbf{V}}^{\lambda}$ and $W$ is a subspace of $H^0(\cL_i)$ of dimension $\kappa_j$.
  \item There is  a compatible extension $\{\widetilde{V}_l\}_{l\leq i}$ of $\textbf{V}_{(i)}$ to $\cL_i$ such that $W$ is the inverse image of $V_j$ in $H^0(\cL_i)$ under the isomorphism $\widetilde{V}_i\rightarrow V_i$.
\end{enumerate}
\end{definition}

We call the $W$ in a pair $(\cL_i, W)$ as above a lifting of $V_j$ to $\cL_i$. For any element $s\in V_j$, the preimage of $s$ via the isomorphism $W\rightarrow V_j$ is called a lifting of $s$ to the level $i$.

In the appendix we also show that $\widetilde{S}_{i,j,\textbf{V}}^{\lambda}$ is a constructible subset of a suitable Grassmann bundle. For the convenience, we set $\widetilde{S}_{0,j,\textbf{V}}^{\lambda}=\{(L,V_j)\}$ for every $\lambda$.

\begin{lemma}\label{dimofbundles}
Let $X_1$ and $Y_1$ be constructible subsets of algebraic varieties $X$ and $Y$ respectively. Let $f: X_1\rightarrow Y_1$ be the restriction of a morphism $g: X\rightarrow Y$. If all the fibers of $f$ are of dimension $d\geq 0$, then $\dim X_1=\dim Y_1+d$.
\end{lemma}

\begin{proof} Since $Y_1$ is a constructible subset of $Y$, we write $Y_1$ as a finite disjoint union of locally closed subset $V_k$ of $Y$. We may assume that all subsets $V_k$ are irreducible. For every $k$, the inverse image $f^{-1}(V_k)$, as the intersection of $g^{-1}(V_k)$ with $X_1$, is a constructible subset of $X$. We thus have $\dim Y_1=\max\limits_{k}\{\dim V_k\}$ and $\dim X_1=\max\limits_{k}\{\dim f^{-1}(V_k)\}$. Hence it is enough to show the statement for the map $f^{-1}(V_k)\rightarrow V_k$ for every $k$. We may thus assume that $Y_1$ is an irreducible algebraic variety.

Consider the stratification $X_1=\coprod\limits_{l=1}^m(W_l)$, where each $W_l$ is a locally closed subset of $X$. For every $l$, the morphism $W_l\rightarrow Y_1$ has fibers of dimension $\leq d$. We get $$\dim W_l\leq \dim f(W_l)+d\leq \dim Y_1+d.$$ This implies that $\dim X_1=\max\limits_{l}\{\dim W_l\}\leq \dim Y_1+d$.

We now prove the other direction of the inequality. Let $\{W_l\}_{l=1,\ldots, m_0}$ be the collection of those $W_l$ that dominates $Y_1$. (This collection is nonempty since otherwise $f^{-1}(y)$ would be empty for a general point $y\in Y_1$.) We choose an open subset $V\subset Y_1$ such that $\dim (W_l\cap f^{-1}(y))$ is constant for $y\in V$ and $l\leq m_0$. There is a subset $W_l$ with $l\leq m_0$ such that $\dim(f^{-1}(y)\cap W_l)=\dim f^{-1}(y)=d$. We obtain that $\dim W_l=d+\dim Y_1$. We thus have
$$\dim X\geq \dim W_1=\dim Y_1+d.$$ This completes the proof.
\end{proof}

\begin{lemma}\label{dimension S}
For a fixed $L\in \pic^d(C)$ with $l=h^0(C,L)$ and a partition $\lambda\in \Lambda_{l,m+1}$, let $\kappa=(\kappa_1,\kappa_2,\ldots,\kappa_m)$ be a signature of length $m$,  with $\kappa_{j}\leq n_{j+1}(\lambda)$ for every $j\leq m$, and $\textbf{V}$ a weak flag of $H^0(L)$ of signature $\kappa$. For every $i$ with $1\leq i\leq m$,  we write $d_i$ for the dimension of the kernel of
$$\mu_{V_i}: V_i\otimes H^0(C,K\otimes L^{-1})\rightarrow H^0(C,K).$$
Then the the following holds:
\begin{enumerate}
  \item $\dim S_{i,\textbf{V}}^{\lambda}-\dim S_{i-1,\textbf{V}}^{\lambda}\leq g+d_i-\kappa_i\cdot (g-d-1+n_i(\lambda))$,
  \item $\dim S_{i,\textbf{V}}^{\lambda}\leq gi-\sum\limits_{j=1}^{i}\{(\kappa_j\cdot(g-d-1+n_{j}(\lambda))-d_j\}$.
\end{enumerate}
\end{lemma}

\begin{proof} Since we fix the partition $\lambda$, we may and will omit the superscript $\lambda$ in the proof.
We apply Lemma \ref{equationoftransitionfunction} to compute the dimension of $S_{i,\textbf{V}}$ inductively on $i$. $S_{0,\textbf{V}}=\{L\}$ implies that $\dim S_{0,\textbf{V}}=0$. Consider the following commutative diagram:
\begin{equation*}
\begin{array}[c]{ccc}
\widetilde{S}_{i,i,\textbf{V}}&\stackrel{h}{\rightarrow}&\widetilde{S}_{i-1,i,\textbf{V}}\\
\downarrow\scriptstyle{\rho_1}&&\downarrow\scriptstyle{\rho_2}\\
S_{i,\textbf{V}}&{\rightarrow}&S_{i-1,\textbf{V}}
\end{array}
\end{equation*}

The horizontal map $h$ maps $(\cL_i,W')$ to $(\cL_{i-1}, W)$, where $W$ is the image of $W'$ under the truncation map $\pi^i_{i-1}: H^0(C, \cL_i)
\rightarrow H^0(C, \cL_{i-1})$.  The vertical map $\rho_1$ is given by mapping $(\cL_i,W')$ to $\cL_i$ and $\rho_2$ is defined similarly.

Let $\cL_i$ be a fixed point in $S_{i,\textbf{V}}$. The fiber of $\rho_1$ over the point $\cL_i$ is the set of linear subspaces $W'\subset H^0(\cL_i)$ that map isomorphically onto $V_i$ via $\pi^i_0: H^0(\cL_i)\rightarrow H^0(L)$.  Let $\{s_{0,k}\}_k$ be a basis of $V_i$. A lifting $W'$ of $V_i$ is determined by the preimage of $s_{0,k}$ in $W'$ for each $k$. By Lemma \ref{filtration lemma}, we see that for every $s\in V_i$, any two liftings of $s$ to the level $i$ differ up to an element of $H^0(\cL_{i-1})$. Therefore, the relative dimension of the map $\rho_1$ is $h^0(\cL_{i-1})\cdot \kappa_i=(\sum\limits_{j=1}^{i}n_j(\lambda))\cdot\kappa_i$. Similarly the relative dimension of the second vertical map $\rho_2$ is $(\sum\limits_{j=1}^{i-1}n_j(\lambda))\cdot \kappa_i$.

Consider the horizontal map $h$. For every element $(\cL_{i-1}, W)\in \widetilde{S}_{i-1,i,\textbf{V}}$, we now give a criterion to decide whether it is in the image of $h$ or not. Fix an element $\mathcal{M}$ in the fiber of $\rho^i_{i-1}: \pic^d(C)_i\rightarrow \pic^d(C)_{i-1}$ over $\cL_{i-1}$. We identify the fiber $(\rho^i_{i-1})^{-1}(\cL_{i-1})$ with $H^1(C,\mathcal{O}_C)$. Let $\{s_{0,k}\}_k$ be a basis of $W$. We donote by $s_{i-1,k}$ the lifting of $s_{0,k}$ to the level $i-1$ in $W$. With the notation in Lemma \ref{equationoftransitionfunction}, every element $s_{i-1,k}=(\sum\limits_{j=0}^{i-1}c_{k,\alpha}^{(j)}t^j)\in H^0(\cL_{i-1})$ has an extension to a section of $\mathcal{M}'\in (\rho^i_{i-1})^{-1}(\cL_{i-1})$ if and only if the following equation holds
\begin{equation}\label{sk}\tag{$\dagger_k$}
\nu(s_{0,k}\otimes [\mathcal{M}'])= {\rm the~ cohomology ~class ~corresponding ~to~ } (- \gamma_\alpha^{-1}(\sum\limits_{j=1}^{i}\phi_\alb^{(j)} c^{(i-j)}_{k,\alpha})).
\end{equation}

Hence $(\cL_{i-1}, W)$ is in the image of $h$ if and only if there is a point $\mathcal{M}'\in (\rho^i_{i-1})^{-1}(\cL_{i-1})$ such that the above identity (\ref{sk}) holds for every $k$. We now assume that $(\cL_{i-1}, W)$ is in the image of $h$ and fix an element $(\mathcal{M}',W')$ in the fiber of $h$ over $(\cL_{i-1}, W)$. The above argument implies that
\begin{equation*}
\rho_1(h^{-1}(\cL_{i-1}, W))=\{[\cL_{i}]\in (\rho^i_{i-1})^{-1}(\cL_{i-1})~|~ \nu(s_{0,k}\otimes ([\cL_{i}]-[\mathcal{M}']))=0 {\rm ~for ~every ~} k\}.
\end{equation*}

By taking the dual linear spaces, we now deduce that $\rho_1(h^{-1}(\cL_{i-1}, W))$ is an affine space consisting of the elements in $H^1(C,\cO_C)$ that annihilate the image of the pairing
$$\mu_{V_i}: V_i\otimes H^0(C,K\otimes L^{-1})\rightarrow H^0(C,K).$$
It follows that $\dim \rho_1(h^{-1}(\cL_{i-1}, W))=g-(\kappa_i\cdot(l-d-1+g)-d_i).$

If $\cL_{i}$ is an element in $\rho_1(h^{-1}(\cL_{i-1}, W))$, then a pair $(\cL_i,W')$ is in the fiber of $h$ over $(\cL_{i-1}, W)$ if and only if the truncation map $H^0(\cL_i)\rightarrow H^0(\cL_{i-1})$ takes $W'$ into $W$. A lifting $W'$ of $W$ is determined by the preimage of $\{s_{i-1,k}\}_k$ in $W'$. By Lemma \ref{filtration lemma}, we see that any two liftings only differ by an element of $H^0(L)$. Hence we deduce that $h^{-1}(\cL_{i-1},W)\cap\rho_1^{-1}(\cL_i)$ is an affine space of dimension $\kappa_i\cdot l$. Thus the dimension of every nonempty fiber of the horizonal map $h$ is $g+d_i-\kappa_i\cdot (l-d-1+g)+\kappa_i\cdot l$.

 By Lemma \ref{dimofbundles}, we have:
 \begin{equation*}
 \begin{array}{cl}
 \dim\widetilde{S}_{i,i,\textbf{V}}=\dim S_{i,\textbf{V}}+\kappa_i\cdot (\sum\limits_{j=1}^{i}n_j(\lambda))\\
 \dim \widetilde{S}_{i,i,\textbf{V}}\leq \widetilde{S}_{i-1,i,\textbf{V}}+g+d_i-\kappa_i\cdot (l-d-1+g)+\kappa_i\cdot l\\
 \dim \widetilde{S}_{i-1,i,\textbf{V}}=\dim S_{i-1,\textbf{V}}+\kappa_i\cdot (\sum\limits_{j=1}^{i-1}n_j(\lambda)))
 \end{array}
 \end{equation*}
It follows that
 $$\dim S_{i,\textbf{V}}-\dim S_{i-1,\textbf{V}}\leq g-\kappa_i\cdot (g-d-1)+d_i-\kappa_i\cdot n_i(\lambda).$$ This proves $(1)$.

Part $(2)$ follows from $(1)$ by induction on $i$, using $\dim S_{0,\textbf{V}}=0$.
\end{proof}

\begin{remark}\label{rmk}
From the proof, we know that the equality in (1) can be achieved if the map $h: \widetilde{S}_{i,i,\textbf{V}}^{\lambda}\rightarrow \widetilde{S}_{i-1,i,\textbf{V}}^{\lambda}$ is a surjection. In fact, we will see that equality can be achieved when we apply the above lemma in the proofs of the main theorems.
\end{remark}

We now prove our first main result.

\noindent{\it Proof} of Theorem \textbf{A}.
Let $C$ be a curve of genus $g\geq 3$. Since we fix the curve $C$, we may and will write $W^r_d$ for $W^r_d(C)$ for every $r$ and $d$. By Remark \ref{smoothness}, we know that $\Theta_{\text{sing}}=W_{g-1}^{1}=\bigcup\limits_{l\geq 2} (W_{g-1}^{l-1}\smallsetminus W_{g-1}^{l})
$. To bound the dimension of $(\pi^{\Theta}_m)^{-1}(\Theta_{\text{sing}})$, it is enough to bound the dimension of $(\pi^{\Theta}_m)^{-1}(W_{g-1}^{l-1}\smallsetminus W_{g-1}^
{l})$ for each $l\geq 2$.

Let $L$ be a point in $W_{g-1}^{l-1}\smallsetminus W_{g-1}^{l}$. We have seen in the proof of Theorem \ref{multiplicity thm} that $\Theta_{m,L}=\pic^{g-1}(C)_{m,L}$ for $m< l $.  Hence $\dim\Theta_{m,L}=mg$ for $m<l$. We now assume that $m\geq l$. Recall that we put $C_{\lambda,m}=\{\cL_m\in \Theta_{m,L}~|~ \cL_m \text{ is of type } \lambda\}$, where $\lambda$ is a partition in $\Lambda_{l,m+1}$. By Lemma \ref{condition}, $\Theta_{m,L}$ is a finite union of $C_{\lambda,m}$, with $\lambda$ satisfying $\sum\limits_{i=1}^{l}{\lambda_i}\geq m+1$. In order to prove the theore, we first bound the dimension of each $C_{\lambda,m}$.

We now fix a partition $\lambda\in \Lambda_{l,m+1}$ with $\sum\limits_{i=1}^{l}{\lambda_i}\geq m+1$. Let $\kappa$ be the signature with $\kappa_i=1$ for every $i\leq \lambda_{l}-1$ and $\kappa_i=0$ for $i\geq \lambda_l$. If $\cL_m\in C_{\lambda}$, we denote by $\cL_i$ the image of $\cL_m$ under $\rho^m_i:\pic^m(C)_m\rightarrow \pic^i(C)_i$ for every $i\leq m$. The definition of $n_k(\lambda)$ implies that $\lambda_l$ is the largest index $k$ such that $n_k(\lambda)\neq 0$. Remark \ref{image} implies that the map $\pi^{\lambda_l-1}_0: H^0(\cL_{\lambda_l-1})\rightarrow H^0(L)$ is nonzero while the map $\pi^{\lambda_l}_0: H^0(\cL_{\lambda_l})\rightarrow H^0(L)$ is zero. Let $W\subset H^0(C,L)$ be the $1$--dimensional subspace in the image of $\pi^{\lambda_l-1}_0$. Consider a weak flag of $H^0(L)$ of signature $\kappa$,
\begin{equation*}
\textbf{V}_W: H^0(C,L)=V_0\supset V= V_1=\cdots V_{\lambda_l-1}=W\supset V_{\lambda_l}=\cdots=V_m=0.
\end{equation*}
Hence $\cL_i$ is in $S_{i,\textbf{V}}^\lambda$ for each $i\leq \lambda_l-1$. We thus conclude that the truncation map $\rho^m_{\lambda_{l}-1}: \pic^{g-1}(C)_{m,L}\rightarrow \pic^{g-1}(C)_{\lambda_{l}-1,L}$
maps $C_{\lambda,m}$ into $S_{\lambda_{l}-1,\kappa}^\lambda$.

Let $\text{Flag}_\kappa$ be the variety parameterizing all weak flags of signature $\kappa$. Let $W$ be a 1-dimensional subspace of $H^0(C,L)$. It defines a weak flag $\textbf{V}_W=\{V_i\}$ of signature $\kappa$. We thus have a bijection between $\text{Flag}_\kappa$ and $\PP(H^0(C,L))$. We now compute the dimension of $S_{\lambda_l-1,\textbf{V}_W}^{\lambda}$.  Let $s_0$ be a nonzero element in $W$. The multiplication map
$$m_{s_0}: H^0(C,K_C\otimes L^{-1})\rightarrow H^0(C,K_C)$$
is always injective. We thus conclude that $W\otimes H^0(C,K_C\otimes L^{-1})\rightarrow H^0(C,K_C)$ is injective. Recall that $d_i$ is the dimension of the kernel of map
$$\mu_{V_i}: V_i\otimes H^0(C,K_C\otimes L^{-1})\rightarrow H^0(C,K_C).$$ We conclude that $d_i=0$ for every $i$ with $0\leq i\leq m$. Moreover, the dual map of $m_{s_0}$, denoted by $m_{s_0}^*: H^1(C,\cO_C)\rightarrow H^1(C,L)$, is a surjection. Lemma \ref{equationoftransitionfunction} implies that
for every $i\leq \lambda_{l}-1$ and every $s_{i-1}\in H^0(\cL_{i-1})$ which is a lifting of $s_0$, there are line bundles $\cL_i$ over $\cL_{i-1}$ such that $s_{i-1}$ can be extended as a section of $\cL_i$.
Therefore, the horizontal map $h: \widetilde{S}_{i,i,\textbf{V}_W}^\lambda\rightarrow \widetilde{S}_{i-1,i,\textbf{V}_W}^\lambda$ is a surjection. By
Lemma \ref{dimension S} and Remark \ref{rmk}, we obtain that for every weak flag $\textbf{V}_W$
$$\dim S_{\lambda_l-1,\textbf{V}_W}^\lambda=(\lambda_l-1)g-\sum\limits_{k=1}^{\lambda_l-1}n_k(\lambda).$$

By Lemma \ref{dimofbundles}, we obtain that $\dim S_{\lambda_{l}-1,\kappa}^{\lambda}\leq \max\limits_{W}\{\dim S_{\lambda_l-1,\textbf{V}_W}^\lambda\}+\dim \PP H^0(L)$, where $W\in H^0(C,L)$. We consider $C_{\lambda,m}$ as a subset of the preimage of $S_{\lambda_{l}-1,\kappa}^{\lambda}$ under the map $\rho^m_{\lambda_l-1}:\pic^{g-1}(C)
_m \rightarrow \pic^{g-1}(C)_{\lambda_l-1}$.

Hence
\begin{equation*}
\begin{array}{cl}
\dim C_{\lambda,m} & \leq  g\cdot (m-\lambda_l+1)+\max\limits_{W}\{\dim S_{\lambda_l-1,\textbf{V}_W}\}+\dim \PP(H^0(L)) \\
&=mg-\sum\limits_{j=1}^{\lambda_l-1}n_j+l-1\\ &=mg-(\sum\limits_{i=1}^{l}\lambda_i-r_{\lambda_l})+l-1
\end{array}
\end{equation*}

Martens' theorem says that for every smooth curve of genus $g\geq 3$, and every $d$ and $r$ with $2\leq d\leq g-1$ and $0< 2r\leq d$, we have $\dim W^r_d(C)\leq
d-2r$. (See \cite{Kem}). For $1\leq m<l$, we have
\begin{equation*}
\begin{array}{cl}
\dim(\pi^{\Theta}_m)^{-1}(W^{l-1}_{g-1}\smallsetminus W^{l}_{g-1})&=\dim{(W^{l-1}_{g-1}\smallsetminus W^{l}_{g-1})}+ mg\\
&\leq g-1-2(l-1)+mg\\
&=(m+1)(g-1)+(m-2(l-1))\\
&\leq (m+1)(g-1)-m
\end{array}
\end{equation*}
For $m\geq l$, we have
\begin{equation}\label{star}\tag{$\ast$}
\begin{array}{cl}
\dim(\pi^{\Theta}_m)^{-1}(W^{l-1}_{g-1}\smallsetminus W^{l}_{g-1})&\leq \max\limits_{\lambda}\left\{\dim(C_\lambda)+g-2l+1\right\}\\
&\leq \max\limits_{\lambda} \left\{(m+1)(g-1)-(\sum\limits_{i=1}^{l}\lambda_i-m-1)-(l-r_{\lambda_l}(\lambda))\right\}
\end{array}
\end{equation}
where $\lambda$ varies over partitions in $\Lambda_{l,m+1}$ with $\sum_{i=1}^{l}\lambda_i\geq m+1$.
We conclude that for every $m$,we have $\dim(\pi^{\Theta}_m)^{-1}(W^{l-1}_{g-1}\smallsetminus W^{l}_{g-1})\leq (m+1)(g-1)$. Furthermore, if the equality is achieved for some $m$, then there is $\lambda\in \Lambda_{l,m+1}$ such that $\sum\limits_{i=1}^{l}\lambda_i=m+1$ and $l=r_{\lambda_l}(\lambda)$, i.e. $\lambda_1=\cdots=\lambda_l$. It follows that for $m$ such that $m+1$ is not divisible by any integer $l\in [2,g-1]$, the set $(\pi^{\Theta}_m)^{-1}(\Theta_{\text{sing}})=\bigcup\limits_{l\geq 2}(\pi^{\Theta}_m)^{-1}(W^{l-1}_{g-1}\smallsetminus W^{l}_{g-1})$  has dimension smaller than $(m+1)(g-1)$. Hence $\Theta_m$ is irreducible for arbitrarily large $m$, which implies that $\Theta_m$ is irreducible for all $m$. (See \cite[Proposition 1.6]{Mus1}.) This implies that
$$\dim (\pi^{\Theta}_m)^{-1}(\Theta_{\text{sing}})\leq (m+1)
(g-1)-1$$
for every $m$.

In order to get the lower bound for $\dim (\pi^{\Theta}_m)^{-1}(\Theta_{\text{sing}})$, we need the following lemma, see \cite[proposition 1.6]{Mus1}.
\begin{lemma}\label{count more}
If $X$ is a locally complete intersection variety of dimension n and $Z\subset X$ is a closed subscheme, then $\dim(\pi^X_{m+1})^{-1}(Z)\geq \dim(\pi^X_{m})^{-1}(Z) +n$ for every $m\geq 1$.
\end{lemma}

If $C$ is a hyperelliptic curve, we show that $\dim (\pi^{\Theta}_m)^{-1}(\Theta_{\text{sing}})=(m+1)(g-1)-1$ by induction on $m\geq 1$. By \cite[\S VI.4]{ACGH}, we know that $\Theta_{\text{sing}}$ has dimension equal to $g-3$, which implies that $\dim (\pi^{\Theta}_{1})^{-1}(\Theta_{\text{sing}})=g-3+g=2(g-1)-1$. We thus have the assertion for $m=1$.  Assume now that the assertion holds for $m-1$.  A repeated application of Lemma \ref{count more} implies that for every $m\geq 1$,
$\dim (\pi^{\Theta}_m)^{-1}(L)\geq (m-1)(g-1)+\dim(\pi_1^{-1})(L)$.
Hence for every $L\in \Theta_{\text{sing}}$, $\dim (\pi^{\Theta}_m)^{-1}(L)\geq (m-1)(g-1)+g$. Therefore
$$\dim (\pi^{\Theta}_m)^{-1}(\Theta_{\text{sing}})\geq\dim \Theta_{\text{sing}}+(m-1)(g-1)+g=(m+1)(g-1)-1.$$
This completes the proof of the theorem for hyperelliptic curves.

We now assume that $C$ is a nonhyperelliptic curve of genus $g$, and show that $\dim (\pi^{\Theta}_{m})^{-1}(\Theta_{\text{sing}})=(m+1)(g-1)-2$ by induction on $m$. By \cite[\S VI.4]{ACGH}, we have $\dim\Theta_{\text{sing}}=g-4$, hence $\dim (\pi^{\Theta}_{1})^{-1}(\Theta_{\text{sing}})=g-4+g=2(g-1)-2$. This proves the assertion for $m=1$. A repeated application of Lemma \ref{count more} implies that for every $L\in \Theta_{\text{sing}}$ and every $m\geq 1$, $\dim (\pi^{\Theta}_m)^{-1}(L)\geq (m-1)(g-1)+\dim(\pi_1^{\Theta})^{-1}(L)$. We thus have $\dim (\pi^{\Theta}_m)^{-1}(\Theta_{\text{sing}})\geq g-4+(m-1)(g-1)+g=(m+1)(g-1)-2$.

In order to finish the proof, it is enough to show that
$$\dim (\pi^{\Theta}_m)^{-1}(\Theta_{\text{sing}})<(m+1)
(g-1)-1$$  for every $m$. Assume there is some $m_{0}$ such that $\dim (\pi^{\Theta}_{m_{0}})^{-1}(\Theta_{\text{sing}})\geq (m_{0}+1)(g-1)-1$. A repeated application of Lemma \ref{count more} implies that for every $m>m_{0}$,
$$\dim (\pi^{\Theta}_{m})^{-1}(\Theta_{\text{sing}})\geq (m-m_0)(g-1)+\dim (\pi^{\Theta}_{m_{0}})^{-1}(\Theta_{\text{sing}})\geq(m+1)(g-1)-1.$$ On the other hand, for nonhyperelliptic curves, Martens' theorem has a better bound, namely $\dim W^{r}_{d}\leq d-2r-1$. By applying it for the theta divisor, we have
$$\dim (W^{l-1}_{g-1}\smallsetminus W^{l}_{g-1})\leq g-2l.$$
Arguing as in (\ref{star}), we obtain
\begin{eqnarray*}
\dim(\pi^{\Theta}_m)^{-1}(W^{l-1}_{g-1}\smallsetminus W^{l}_{g-1})\leq \max\limits_{\lambda}\left\{(m+1)(g-1)-(\sum\limits_{i=1}^{l}\lambda_i-m-1)-(l-r_{\lambda_l}(\lambda))-1\right\}
\end{eqnarray*}
where $\lambda$ varies over partitions in $\Lambda_{l,m=1}$ with $\sum_{i=1}^l\lambda_i\geq m+1$.
It follows that unless there is a $\lambda\in \Lambda_{l,m+1}$ with $\sum\limits_{i=1}^{l}\lambda_i=m+1$ and $r_{\lambda_l}(\lambda)=l$, we have $$\dim(\pi^{\Theta}_m)^{-1}(\Theta_{\text{sing}})< (m+1)(g-1)-1.$$ Therefore this holds for every $m$ such that $m+1$ is not divisible by any integer $2\leq l\leq g-1$. Since there are arbitrarily large such $m$, we obtain a contradiction.
\hfill{$\Box$}

In \cite{Mus1}, Musta\c{t}\v{a} describes complete intersection rational singularities in terms of jet schemes as follows.
 If X is a local complete intersection variety of dimension $n$ over $k$, then the following are equivalent:
\begin{enumerate}
\item[(i)] $X$ has rational singularities.

\item[(ii)] $X$ has canonical singularities.

\item[(iii)] $X_m$ is irreducible for each $m$.

\item[(iv)] $\dim \pi_m^{-1}(X_{\text{sing}})< n(m+1)$ for every $m$.
\end{enumerate}
The equivalence of the first two parts is due to Elkik, see \cite{Elk}.
Note also that by Theorem 3.3 in \cite{EMY}, for a reduced irreducible divisor $D$ on a smooth variety $X$ of dimension $n$,
the following are equivalent,

\begin{enumerate}
\item[(i)]  The jet scheme $D_m$ is a normal variety for every $m$.\\
\item[(ii)]  $D$ has terminal singularities. \\
\item[(iii)]  For every $m$, $\dim (\pi^{D}_m)^{-1}(D_{\text{sing}})\leq (m+1)(n-1)-2$.
\end{enumerate}
Applying these two results to the theta divisor, we obtain the following result concerning the singularities of this variety.

\begin{corollary}\label{terminal and ratioanl}
Let  $C$ be a smooth projective curve  of genus $g\geq 3$ over $k$.
The theta divisor has terminal singularities if $C$ is a nonhyperelliptic  curve.
If $C$ is hyperelliptic,  then the theta divisor has canonical non-terminal singularities. In particular, the theta divisor has rational singularities for every smooth curve.
\end{corollary}
We now apply the above ideas to compute the log canonical threshold of the pair
$(\pic^d(C), W^r_d(C))$ at a point $L\in W^r_d(C)$, where $C$ is general in the moduli space of curves.

In \cite[Corollary 3.6]{Mmus2}, one gives the following formula for the log canonical threshold of a pair in terms of the dimensions of the jet schemes. If $Y\subset X$ is a closed subscheme and $Z\subset X$ is a nonempty closed subset, then the log canonical threshold of the pair $(X,Y)$  at $Z$ is given by
$$\lct_Z(X,Y)=\dim X-\sup\limits_{m\geq 0}\frac{\dim (\pi^Y_m)^{-1}(Y\cap Z)}{m+1}.$$

For every $L\in W^r_d(C)$, the above formula implies that $$\lct_L(\pic^d(C),W^r_d(C))=g-\sup\limits_{m\geq 0} \frac{\dim W^r_d(C)_{m,L}}{m+1}.$$ Our main goal is to estimate the dimension of $W^r_d(C)_{m,L}$ for each $m$.

We now turn to the proof of Theorem \textbf{B}.
Let $C$ be a general smooth projective curve of genus $g$ and let $L$ be a line bundle on $C$. The generality assumption on $C$ implies that the natural pairing
$$\mu_0: H^0(C, L)\otimes H^0(C,K_C\otimes L^{-1})\rightarrow H^0(C,K_C)$$
is injective for every $L$. This was stated by Petri and first proved by Gieseker \cite{Gie}.

Before proving the theorem, we need to prove an identity for every partition as preparation.

\begin{lemma}\label{idforpart}
Let $\lambda\in \Lambda_{l,m+1}$ and $\lambda_0=0$. We now prove that
$$\sum\limits_{i=1}^{\lambda_l}
n_i^2(\lambda)=\sum\limits_{i=1}^{l}(l-i+1)^2(\lambda_i-\lambda_{i-1}).$$
\end{lemma}

\begin{proof}
Given a partition $\lambda\in \Lambda_{l,m+1}$, we may write it as:
\begin{eqnarray*}
1\leq \lambda_1=\cdots=\lambda_{m_1}<\lambda_{m_1+1}=\cdots=\lambda_{m_2}<\lambda_{m_2+1}\cdots \lambda_{m_k}<\lambda_{m_{k}+1}=\cdots=\lambda_l.
\end{eqnarray*}
For simplicity, we write $n_i$ for $n_{i}(\lambda)$. It is easy to see that
\begin{eqnarray*}
n_1&=&\cdots=n_{\lambda_{m_1}}=l\\
n_{\lambda_{m_1}+1}&=&\cdots=n_{\lambda_{m_2}}=l-m_1\\
~~~\cdots\\
n_{\lambda_{m_{k}}+1}&=&\cdots=n_{\lambda_{l}}=l-m_{k-1}
\end{eqnarray*}

This implies that
\begin{eqnarray*}
\sum\limits_{i=1}^{l}(l-i+1)^2(\lambda_i-\lambda_{i-1})&=&
l^2(\lambda_{1})+(l-m_1)^2(\lambda_{m_1+1}-\lambda_{m_1})+\cdots+(l-m_{k})^2(\lambda_{m_k+1}-\lambda_{m_k})\\
&=&l^2(\lambda_{m_1})+(l-m_1)^2(\lambda_{m_2}-\lambda_{m_1})+\cdots+(l-m_{k})^2(\lambda_{l}-\lambda_{m_k})\\
&=&\sum\limits_{i=1}^{\lambda_{m_1}}l^2+\sum\limits_{i=\lambda_{m_1+1}}^{\lambda_{m_2}}n_i^2+\cdots+\sum\limits_{i={\lambda_{m_{k}+1}}}^{\lambda_l}n_i^2\\
&=&\sum\limits_{i=1}^{\lambda_l}n_i^2=\sum\limits_{i=1}^{\lambda_l}n_i^2(\lambda)
\end{eqnarray*}
\end{proof}

\noindent{\it Proof} of Theorem B.
Let C be a general smooth projective curve in the sense of Petri and Gieseker. Let $L$ be a line bundle in $W^r_d(C)$ with $l=h^0(C,L)\geq r+1$.
Since we are only interested in the asymptotic behavior of $W^r_d(C)_{m,L}$, we can assume that $m\geq l$. By Lemma \ref{conditionw} we have a stratification $W^r_d(C)_{m,L}=
\bigcup C_{\lambda,m}$, where $\lambda\in \Lambda_{l,m+1}$ are taken over all length $l$ partitions satisfying  $\sum\limits_
{i=1}^{l-r }{\lambda_i}\geq m+1$.

We now fix such a partition $\lambda$. Let $\kappa$ be a signature with $\kappa_i=n_{i+1}(\lambda)$ for $i$ with $1\leq i\leq m$. For every $\cL_m\in C_{\lambda,m}$, we denote by $\cL_i$ the image of $\cL_m$ under the truncation $\rho^m_{i}: \pic^d(C)_m\rightarrow \pic^d(C)_i$. The images $V_i$ of the maps $\pi^i_0: H^0({\cL_i})\rightarrow H^0(C,L)$ give a weak flag $\textbf{V}_{\cL_m}$. By Remark \ref{image}, we obtain $\dim V_i=h^0(\cL_i)-h^0(\cL_{i-1})=n_{i+1}(\lambda)$. Therefore $\textbf{V}_{\cL_m}$ is a weak flag of signature $\kappa$. The image of $\cL_m$ in $\pic^d(C)_{\lambda_l-1}$ is in $S_{\lambda_l-1,\textbf{V}_{\cL_m}}^\lambda$. Hence the truncation map $\rho_{\lambda_{l}-1}^{m}: \pic^d(C)_{m} \rightarrow \pic^d(C)_{\lambda_{l}-1}$ maps $C_{\lambda,m}$ to $S_{\lambda_l-1, \kappa}^\lambda=\bigcup\limits_{\textbf{V}}S_{\lambda_l-1,\textbf{V}}^\lambda$, where $\textbf{V}$ varies over all weak flags of signature $\kappa$. The key step is to compute the dimension of $S_{\lambda_l-1,\textbf{V}}^\lambda$ for each $\textbf{V}$. We keep the notation in the proof of Lemma \ref{dimension S}.

The fact that the canonical pairing
$$\mu_0: H^0(C, L)\otimes H^0(C,K_C\otimes L^{-1})\rightarrow H^0(C,K_C)$$
 is injective implies that all restrictions $\mu_{V_i}: V_i\otimes H^0(C,K_C\otimes L^{-1})\rightarrow H^0(C,K_C)$ are injective. Hence $d_i=\dim \ker \mu_{V_i}$ is zero for every weak flag $\textbf{V}$ of $H^0(L)$ of signature $\kappa$.

We now show that if the canonical pairing $\mu_0$ is injective, then all horizontal maps $h: \widetilde{S}_{i,i,\textbf{V}}^\lambda\rightarrow \widetilde{S}_{i-1,i, \textbf{V}}^\lambda$ in the proof of Lemma \ref{dimension S} are surjective. Let $(\cL_{i-1}, W)$ be an element in $\widetilde{S}_{i-1,i,\textbf{V}}^\lambda$. Given a point $\mathcal{M}$ in the fiber of $\rho^i_{i-1}: \pic^d(C)_i\rightarrow \pic^d(C)_{i-1}$ over $\cL_{i-1}$, we get an isomorphism $(\rho^i_{i-1})^{-1}(\cL_{i-1})\cong H^1(C,\mathcal{O}_C)$.  Let $\{s_{0,p}\}_p$ be a basis of $V_i$, and $s_{i-1,p}$ the lifting of $s_{0,p}$ to the level $i-1$ in $W$. It is easy to see that $(\cL_{i-1}, W)$ is in the image of $h$ if and only if there is an element $\cL_i\in (\rho^i_{i-1})^{-1}(\cL_{i-1})$ such that for every $p$, the section $s_{i-1,p}$ has an extension to a section of $\cL_{i}$. By Lemma \ref{equationoftransitionfunction}, we deduce that for every $p$, $s_{i-1,p}$ has an extension to a section of $\cL_i$ if and only if the equation of the following form holds:
\begin{equation}\label{ssk}\tag{$\diamond_p$}
\nu(s_{0,p}\otimes [\cL_i])=\tau_p.
\end{equation}
where
$\tau_p$ is a cohomology class in $H^1(C,L)$ determined by the section $s_{i-1,p}$.
In order to prove that $(\cL_{i-1},W)$ is in the image of $h$, it suffices to show 
the existence of an element $\cL_i\in (\rho^i_{i-1})^{-1}(\cL_{i-1})$ such that the equation (\ref{ssk}) holds for every $p$.

Recall that $H^1(C,\mathcal{O}_C)$ is the dual space of $H^0(C,K_C)$, hence we identify $[\cL_i]$ with a linear map $H^0(C,K_C)\rightarrow k$. By the duality between $H^1(C,L)$ and $H^0(C,K_C\otimes L^{-1})$, we identify $\tau_p$ with a linear map $H^0(C,K_C\otimes L^{-1})\rightarrow k$. For every $p$, there is a map $m_{s_{0,p}}: H^0(C,K_C\otimes L^{-1})\rightarrow H^0(C,K_C)$  taking $\gamma \in H^0(C,K_C\otimes L^{-1})$ to $\mu_{0}(s_{0,p}\otimes \gamma)$. Hence the equation (\ref{ssk}) holds for $\cL_i$ for all $p$ if and only if the composition map
$$H^0(C,K_C\otimes L^{-1})\stackrel{m_{s_{0,p}}}\hookrightarrow H^0(C,K_C)\stackrel{[\cL_i]}\rightarrow k$$
is equal to $\tau_p$ for all $p$.

Let $A_p$ be the image of $m_{s_{0,p}}$. The fact that $V_i \otimes H^0(C,K_C\otimes L^{-1})\rightarrow H^0(C,K_C)$ is injective implies that the sum $\sum\limits_{p}A_p$ in $H^0(C,K_C)$ is a direct sum. We conclude that there is a \u{C}ech  cohomology classes $[\cL_i]$ satisfying (\ref{ssk}) for all $p$. Therefore $(\cL_{i-1}, W)$ is in the image of $h$.

Applying Lemma \ref{dimension S} and Remark \ref{rmk} to the case $i={\lambda_l-1}$, we obtain
\begin{equation*}
  \dim S_{\lambda_l-1,\textbf{V}}^{\lambda}=g(\lambda_l-1)-\sum\limits_{i=2}^{\lambda_l}n_{i}(\lambda)(g-d-1)-
  \sum\limits_{i=1}^{\lambda_l-1}n_{i+1}(\lambda)n_{i}(\lambda).
\end{equation*}
Recall that $\text{Flag}_\kappa$ is the variety parameterizing all weak flag variety of signature $\kappa$. We denote by $D_{\kappa}$ the dimension of $\text{Flag}_{\kappa}$. It is easy to see that $\text{Flag}_\kappa$ is exactly the usual flag variety of signature $\kappa'$ where $\kappa'$ is the longest decreasing subsequence of $\kappa$. Since $k_1=n_2(\lambda)\leq n_1(\lambda)=l=H^0(C,L)$, there are only finitely many ways to get strictly decreasing sequence with values $\leq l$ and length $\leq l$. There are thus only finitely many integers $D_{\kappa}$. Let $K_1$ be the maximal value among these numbers. Clearly $K_1$ only depends on $l$. In particular, it is independent on $m$, hence $\lim\limits_{m\rightarrow \infty}\frac{K_1}{m}=0$.

Note that $S_{\lambda_l-1,\kappa}^\lambda=\bigcup\limits_{\textbf{V}}S_{\lambda_l-1,\textbf{V}}^\lambda$, where $\textbf{V}$ varies over all weak flags in $\text{Flag}_\kappa$. We thus have $\dim S_{\lambda_l-1,\kappa}^\lambda\leq \max\limits_{\textbf{V}}\{\dim S_{\lambda_l-1,\textbf{V}}^\lambda\}+K_1$.  We have seen that $\rho^m_{\lambda_l-1}(C_{\lambda,m})\subset S_{\lambda_l-1,\kappa}^\lambda$, hence $\dim C_{\lambda,m}\leq(m-\lambda_{l}+1)g+ \dim S_{\lambda_l-1,\kappa}^\lambda$. For each $m\geq l$, we thus have
\begin{eqnarray*}
  \codim(\pic^d(C)_{m,L}, W^r_d(C)_{m,L}) &=& \min\limits_\lambda\left\{mg-\dim C_{\lambda,m}\right\}\\
  &\geq& \min\limits_\lambda\left\{\sum\limits_{i=2}^{\lambda_l}n_i(\lambda)(g-d-1)+ \sum\limits_{i=1}^{\lambda_l-1}n_{i}(\lambda)n_{i+1}(\lambda)-K_1\right\}\\
   &=& \min\limits_\lambda\left\{(\sum\limits_{i=1}^{l}\lambda_i-l)(g-d-1)+\sum\limits_{i=1}^{\lambda_l-1}n_{i}(\lambda)n_{i+1}(\lambda)-K_1\right\}
\end{eqnarray*}
where $\lambda$ varies over the partitions in $\Lambda_{l,m+1}$ with $\sum\limits_
{i=1}^{l-r}{\lambda_i}\geq m+1$.

Note that $\sum\limits_{i=1}^{\lambda_l-1}n_i(\lambda)n_{i+1}(\lambda)\geq \sum\limits_{i=1}^{\lambda_l}n_i^2(\lambda)-l^2$, and since
$\lim\limits_{m\rightarrow \infty}\frac{l^2}{m}=0=\lim\limits_{m\rightarrow \infty}\frac{K_1}{m}$, we obtain

\begin{eqnarray*}
&&\inf\limits_{m\rightarrow \infty} \frac{\codim (\pic^d(C)_{m,L}, W^r_d(C)_{m,L})}{m+1} \\
&\geq& \inf\limits_{m\rightarrow \infty}\min\limits_\lambda\left\{\frac{1}{m+1}\left((\sum\limits_{i=1}^{l}\lambda_i-l)(g-d-1)+\sum\limits_{i=1}^{\lambda_l-1}n_{i+1}(\lambda)n_i(\lambda)-K_1\right)\right\}\\
&\geq&\inf\limits_{m\rightarrow \infty}\min\limits_\lambda\left\{\frac{1}{m+1}\left((\sum\limits_{i=1}^{l}\lambda_i)
(g-d-1)+\sum\limits_{i=1}^{\lambda_l}
n_i^2(\lambda)\right)\right\}\\
\end{eqnarray*}

By Lemma \ref{idforpart}, we thus obtain
\begin{eqnarray*}
&&\inf\limits_{m\rightarrow \infty} \frac{\codim (\pic^d(C)_{m,L}, W^r_d(C)_{m,L})}{m+1}\\
&\geq & \inf\limits_{m\rightarrow \infty}\min\limits_\lambda\left\{\frac{1}{m+1} \left(\sum\limits_{i=1}^{l}(l-i+1)(\lambda_i-\lambda_{i-1})(g-d-1)+ \sum\limits_{i=1}^{l}(l-i+1)^2(\lambda_i-\lambda_{i-1})\right)  \right\}\\
&=&\inf\limits_{m\rightarrow \infty}\min\limits_\lambda\left\{\frac{1}{m+1}\sum\limits_{i=1}^{l}(\lambda_i-\lambda_{i-1})(g-d+l-i)(l-i+1)
\right\}
\end{eqnarray*}

For every $i$ with $1\leq i\leq l$, let $x_i=\lambda_i-\lambda_{i-1}$. Consider a linear function of the form $\sum\limits_{i=1}^{l}b_ix_i$ with $b_i\geq 0$, defined over the region $$\{(x_1,\cdots,x_l)\in \mathbb{R}^l~ |~ x_i\geq 0 {\rm ~for~every ~} i, ~\sum\limits_{i=1}^{l-r}(l-i-r+1)x_i\geq m+1\}.$$ The minimum value of this function is achieved at the vertices of this region, i.e. the points with all the $x_i$ but one equal to $0$ and $\sum\limits_{i=1}^{l-r}(l-i-r+1)x_i=m+1$.

We thus have
\begin{equation}\label{rhs}\tag{$\sharp$}
\lct_L(\pic^d(C),W^r_d(C))\geq\min\limits_{i=1}^{l-r}\left\{\frac{(l+1-i)(g-d-i+l)}{l+1-r-i}\right\}
\end{equation}

On the other hand, recall that one can locally define a map from $\pic^d(C)$ to a variety of matrices $M_{(d+e+1-g)\times e}$ such that $W^r_d(C)$ is the pull back of a suitable generic determinantal variety $Y$ defined by $e+d+1-g-r$ minors. Let $\Phi_{L}$ be the image of $L$. The right hand side in (\ref{rhs}) is the log canonical threshold of the pair $(M_{(d+e+1-g)\times e},Y)$ at the point $\Phi_{L}$ (for the formula of log canonical threshold of a generic determinantal variety, see \cite[Theorem 3.5.7]{Doc}). We thus have $\lct_L(\pic^d(C),W^r_d(C)) \leq \lct_{\Phi_{L}}(M_{(d+e+1-g)\times e},Y)$, by \cite[Example 9.5.8]{Lar}, which completes the proof.
\hfill{$\Box$}

\section{Appendix}
Let $L$ be a line bundle in $\pic^d(C)$ with $l=h^0(C,L)$. In this section, we are going to show that the subsets $S_{i,\textbf{V}}^{\lambda}$ and $S_{i,\kappa}^{\lambda}$ of $\pic^d(C)_{i,L}$ defined in section $2$ are constructible subsets of $\pic^d(C)_{i,L}$. The key point is to realize $\widetilde{S}_{i,j,\textbf{V}}^{\lambda}$ as a constructible subset of a suitable product of Grassmann bundles.

Let $X$ be a scheme and $E$ a vector bundle of rank $n$ over $X$. For every $d\leq n$, we denote by $Gr(d,E)$ the Grassmann bundle of $d$-dimensional subspaces in $E$ and by $\pi$ the projection morphism from $Gr(d,E)$ to $X$. We write elements in $Gr(d,E)$ as pairs $(x,W)$ where $x$ is a point in $X$ and $W$ is a dimension $d$ subspace of $E_x$.

\begin{lemma}\label{degeneracy}
If $\Phi: E\rightarrow F$ is a homomorphism of vector bundles on the scheme $X$, then we have
\begin{enumerate}
\item The subset $I_{\Phi}:=\{x\in X~|~ \Phi_x: E_x\rightarrow F_x ~\text{is an injection}\}$ is an open subset of $X$.
\item If $H$ is a subbundle of $F$, then the set $M^{\Phi}_{H}:=\{x~|~\Phi_x(E_x)\subset H_x\}$ is a closed subset of $X$.
\end{enumerate}
\end{lemma}

The proof of Lemma \ref{degeneracy} is standard, so we leave it to the reader.

Recall that $\mathcal{P}$ is a Poincar\'{e} line bundle on $\pic^d(C)\times C$. From the definition of jet schemes, we have $\pic^d(C)_m\times C_m\cong (\pic^d(C)\times C)_m\cong \Hom(T_m, \pic^d(C)\times C)$. By the adjunction (\ref{adjunction}) in section $1$ for $Y=\pic^d(C)_m\times C_m$ and $X=\pic^d(C)\times C$, the identity map of $\pic^d(C)_m\times C_m$ gives an evaluation morphism $\pic^d(C)_m\times C_m\times T_m\xrightarrow{\Xi} \pic^d(C)\times C$. For every $m$, we also have a morphism $C\xrightarrow{\gamma_m} C_m$ that takes a point to the corresponding constant jet. We have the composition map
$$\eta: \pic^d(C)_m\times C\times T_m\xrightarrow{\id\times \gamma_m \times \id} \pic^d(C)_m\times C_m\times T_m\xrightarrow{\Xi} \pic^d(C)\times C.$$
We denote by $\mathcal{B}_m$ the pull back of the line bundle $\mathcal{P}$ to $\pic^d(C)_m\times C\times T_m$ via $\eta$.

Recall that for every partition $\lambda$ in $\Lambda_{l,m+1}$, $C_{\lambda,m}$ is the locally closed subset
$$\{\cL_m\in \pic^d(C)_{m,L}~|~ \cL_m \text{ is of type } \lambda\}.$$
For every $0\leq i\leq m$, there is a natural map $\Lambda_{l,m+1}\rightarrow \Lambda_{l,i+1}$ mapping $\lambda$ to $\overline{\lambda}$ where $\overline{\lambda}_k=\min\{\lambda_k, i+1\}$ for each $k\leq l$. We have seen that $\rho^m_i: \pic^d(C)_{m,L} \rightarrow \pic^d(C)_{i,L}$ maps $C_{\lambda,m}$ to $C_{\overline{\lambda},i}$.

We now fix a partition $\lambda\in \Lambda_{l,m+1}$. We denote by $\mathcal{B}_{\lambda,m}$ the restriction of $\mathcal{B}_m$ to the subscheme $C_{\lambda,m}\times C\times T_m$, where on $C_{\lambda,m}$ we consider the reduced scheme structure. We denote by $p_1$ the projection to the first factor $\pic^d(C)_{m,L}\times C\times T_m\rightarrow \pic^d(C)_{m,L}$. It is easy to check that for every $\cL_m\in \pic^d(C)_{m,L}$ corresponding to a morphism $f: T_m\rightarrow \pic^d(C)$, the restriction of $\mathcal{B}_{m}$ to the fiber of $p_1^{-1}(\cL_m)\cong C\times T_m$ is $(f\times \id_C)^*(\mathcal{P})\cong\cL_m$.

Recall that for every $i$ with $0\leq i\leq m$, there is a closed embedding $\iota^m_i: T_i\hookrightarrow T_m$. Let
$$\nu^m_i: C_{\lambda,m}\times C\times T_i\hookrightarrow C_{\lambda,m}\times C\times T_m$$
 be the induced embedding. Let $\mathcal{D}_{\lambda,i}$ be the sheaf ${p_1}_*(\nu^m_i)_*(\nu^m_i)^*(\mathcal{B}_{\lambda,m})$ on $C_{\lambda,m}$. Consider the function $C_{\lambda,m}\rightarrow \ZZ$ that takes $\cL_m$ to $h^0(C\times T_i,\cL_i)$, where $\cL_i$ is the image of $\cL_m$ in $\pic^d(C)_i\cong \pic^d(C\times T_i)$. Lemma \ref{dimension lemma} implies that this function is constant on $C_{\lambda,m}$. By the Base Change Theorem, we deduce that $\mathcal{D}_{\lambda,i}$ is a locally free sheaf of rank $\sum\limits_{j=1}^{i+1}n_j(\lambda)$ on $C_{\lambda,m}$, whose fiber over a point $\cL_m$ is $H^0(C,L_i)$.
 For every $i$ and $j$ with $0\leq j\leq i\leq m$, the embedding map $\nu^m_j$ factors through $\nu^m_{i}$. We thus have a natural morphism of sheaves $$(\nu^m_{i})_*(\nu^m_{i})^*(\mathcal{B}_{\lambda,m})\rightarrow(\nu^m_j)_*(\nu^m_j)^*(\mathcal{B}_{\lambda,m})$$
 on $C_{\lambda,m}\times C\times T_{m}$. Applying $(p_1)_*$ to it, we obtain a vector bundle map
 $$\Phi^{i}_j:\mathcal{D}_{\lambda,i}\rightarrow \mathcal{D}_{\lambda,j}$$
 on $C_{\lambda,m}$ whose restriction to the fiber over $\{\cL_m\}$ is the truncation map
 $$\pi^{i}_j: H^0(\cL_{i})\rightarrow H^0(\cL_j).$$

For a fixed partition $\lambda\in \Lambda_{l,m+1}$, we consider $\kappa=(\kappa_1, \cdots, \kappa_m)$ a signature with $k_{j}\leq n_{j+1}(\lambda)$ for every
$j\leq m$. For every $i\leq m$, a point in the fiber product of Grassmann bundles
$$\mathcal{G}_{\lambda,i,\kappa}:=Gr(\kappa_1, \mathcal{D}_{\lambda,1})\times_{C_{\lambda,m}} \cdots \times_{C_{\lambda,m}}  Gr(\kappa_i, \mathcal{D}_{\lambda,i})$$
over $C_{\lambda,m}$ is written as an $(m+1)$--tuple $(\cL_m;\widetilde{V}_1,\cdots, \widetilde{V}_i)$, where $\cL_m\in C_{\lambda,m}$ and $\widetilde{V}_j$ is a dimension $\kappa_j$ subspace of $(\mathcal{D}_{\lambda,j})|_{\cL_m}\cong H^0(\cL_j)$ for every $j\leq i$. For every weak flag $\textbf{V}$ of $H^0(C,L)$ of signature $\kappa$, we denote by $\cP_{m,i,\textbf{V}}^\lambda$ the subset of points $(\cL_m;\widetilde{V}_1,\cdots, \widetilde{V}_i)\in \mathcal{G}_{\lambda,i,\kappa}$, where $\cL_m\in C_{\lambda,m}$ and $\{\widetilde{V}_1,\ldots, \widetilde{V}_i\}$ is a compatible extension of $\textbf{V}_{(i)}$ to the line bundle $\cL_i$. We also write $\cP_{m,i,\kappa}^\lambda$ for $\bigcup\limits_{\textbf{V}}\cP_{m,i,\textbf{V}}^\lambda$, where $\textbf{V}$ varies over all weak flags of $H^0(C,L)$ of signature $\kappa$.

Recall that $\text{Flag}_{\kappa}$ is the variety parameterizing weak flags of $H^0(C,L)$ of signature $\kappa$. We denote by $\widetilde{\mathcal{P}}_{m,i,\kappa}^\lambda$ the subset of points
$$(\cL_m;\widetilde{V}_1,\cdots, \widetilde{V}_i; \textbf{V}')\in \mathcal{G}_{\lambda,i,\kappa}\times \text{Flag}_{\kappa}$$
where $\textbf{V}'\in \text{Flag}_{\kappa}$ and $(\cL_m;\widetilde{V}_1,\cdots, \widetilde{V}_i)\in \mathcal{P}_{m,i,\textbf{V}'}^\lambda$.

\begin{lemma}\label{consofgf}
Let $\lambda\in \Lambda_{l,m+1}$ and $\kappa$ be a signature of length $m$ with $\kappa_j\leq n_{j+1}(\lambda)$ for every $1\leq j\leq m$. Then for every $i$ with $1\leq i\leq m$, $\widetilde{\mathcal{P}}_{m,i,\kappa}^\lambda$ is a constructible subsets of $\mathcal{G}_{\lambda,i,\kappa}\times \text{Flag}_{\kappa}$.
\end{lemma}

\begin{proof}
For simplicity, we write $X$ for the scheme $\mathcal{G}_{\lambda,i,\kappa}\times \text{Flag}_{\kappa}$. For $j$ with $1\leq j\leq i$, we denote by $p_j$ the projection $X$ onto $Gr(\kappa_j,\mathcal{D}_{\lambda,j})$ and by $p_{i+1}$ the projection of $X$ onto $\text{Flag}_{\kappa}$. For a fixed $j$ with $1\leq j\leq i$, we denote by $q_j: Gr(\kappa_j,\mathcal{D}_{\lambda,j})\rightarrow C_{\lambda,m}$. The composition map
$$X\xrightarrow{p_j} Gr(\kappa_j,\mathcal{D}_{\lambda,j})\xrightarrow{q_j} C_{\lambda,m}$$ does not depend on a particular choice of $j$ for $j\leq i$. We denote it by $\chi$.

For every $j$ with $1\leq j\leq i$, we denote by $T_j$ the tautological subbundle of $q_j^*(\mathcal{D}_{\lambda,j})$ on $Gr(\kappa_j,\mathcal{D}_{\lambda,j})$. Let $\mathcal{T}_j=p_j^*{T_j}$ and $\mathcal{F}_j$ be the vector bundle  $p_j^*{q_j^*(\mathcal{D}_{\lambda,j})}=\chi^*(\mathcal{D}_{\lambda,j})$. Hence $\mathcal{T}_j$ is a subbundle of $\mathcal{F}_j$ for each $j$. Over a point $x=(\cL_m;\widetilde{V}_1,\cdots, \widetilde{V}_i; \textbf{V}')\in X$, we have $\mathcal{T}_{j,x}=\widetilde{V}_{j}$ and $\mathcal{F}_{j,x}$ is $H^0(\cL_j)$ where $\cL_j$ is the image of $\cL_m$ under $\pic^d(C)_{m,L}\rightarrow \pic^d(C)_{j,L}$. For every $k$ and $j$ with $0\leq k\leq j\leq i$, we write $\Psi^j_k$ for the composition $\mathcal{T}_j\hookrightarrow \mathcal{F}_j\rightarrow\mathcal{F}_k$.

Let $R_1\supseteq R_2\supseteq\cdots \supseteq R_m$ be the tautological flag bundles on $\text{Flag}_{\kappa}$, where the fiber of $R_j$ over a point $\textbf{V}'=\{V'_n\}_n$ in $\text{Flag}_{\kappa}$ is $V'_j$. We write $\mathcal{R}_j$ for the pull back of $R_j$ via $p_{i+1}:X\rightarrow \text{Flag}_{\kappa}$.
Over a point $x=(\cL_m;\widetilde{V}_1,\cdots, \widetilde{V}_i; \textbf{V}')\in X$, where $\textbf{V}'=\{V'_j\}\in \text{Flag}_\kappa$ we have
$\mathcal{R}_{j,x}=V'_j$. Note that $\mathcal{D}_0$ is the trivial vector bundle on $C_{\lambda,m}$ with fiber $H^0(C,L)$. Hence $\mathcal{F}_0$ is a trivial bundle on $X$ with fiber $H^0(C,L)$. It implies that $\mathcal{R}_{j}$ is a subbundle of $\mathcal{F}_0$.

With the notation in Lemma \ref{degeneracy}, we have
\begin{eqnarray*}
\widetilde{\mathcal{P}}_{m,i,\kappa}^\lambda=\bigcap\limits_{j=1}^{i} (I_{\Psi^j_0}\cap M^{\Psi^j_{j-1}}_{\mathcal{T}_{j-1}}\cap M^{\Psi^j_0}_{\mathcal{R}_j})
\end{eqnarray*}
This completes the proof.
\end{proof}

\begin{corollary}\label{consofp}
With the notation in Lemma \ref{consofgf}, let $\textbf{V}$ be a weak flag of $H^0(C,L)$ of signature $\kappa$. For every $i$ with $1\leq i\leq m$, $\cP_{m,i,\kappa}^\lambda$ and $\cP_{m,i,\textbf{V}}^\lambda$ are both constructible subsets of $\mathcal{G}_{\lambda,i,\kappa}$.
\end{corollary}

\begin{proof}
We denote by $pr_1$ and $pr_2$ the projections of $\mathcal{G}_{\lambda,i,\kappa}\times \text{Flag}_{\kappa}$ onto $\mathcal{G}_{\lambda,i,\kappa}$ and $\text{Flag}_{\kappa}$, respectively. We thus deduce that $\cP_{m,i,\kappa}^\lambda$, as the image of $\widetilde{\mathcal{P}}_{m,i,\kappa}^\lambda$ under $pr_1$, is a constructible subset of $\mathcal{G}_{\lambda,i,\kappa}$. It is clear that $\cP_{m,i,\textbf{V}}^\lambda$ is the image of $\widetilde{\mathcal{P}}_{m,i,\kappa}^\lambda\cap pr_2^{-1}(\textbf{V})$ under $pr_1$. Lemma \ref{consofgf} implies that $\widetilde{\mathcal{P}}_{m,i,\kappa}^\lambda\cap pr_2^{-1}(\textbf{V})$ is a constructible subset of $\mathcal{G}_{\lambda,i,\kappa}\times \text{Flag}_{\kappa}$. This completes the proof.
\end{proof}

\begin{corollary}
Let $\kappa$ be a signature of length $m$ with $k_j\leq n_{j+1}(\lambda)$ for every $j\leq m$, and $\textbf{V}\in \text{Flag}_\kappa$. For every $i$ with $1\leq i\leq m$, the subsets $S_{i,\kappa}^{\lambda}$ and $S_{i,\textbf{V}}^{\lambda}$ are constructible subset of $\pic^d(C)_{i,L}$.
\end{corollary}

\begin{proof}
For a fixed $i$, let $\overline{\kappa}$ be a signature of length $i$ such that $\overline{\kappa}_j=\kappa_j$ for every $j\leq i$. Recall that $\overline{\lambda}$ is the image of $\lambda$ under $\Lambda_{l,m+1}\rightarrow \Lambda_{l,i+1}$.  By the definition of $S_{i,\kappa}^{\lambda}$, we have $S_{i,\kappa}^{\lambda}=S_{i,\overline{\kappa}}^{\overline{\lambda}}$ for every $i\leq m$. It suffice to prove the assertion in case $i=m$.

Recall that $\chi$ is the morphism of projection $\mathcal{G}_{\lambda,m,\kappa}\rightarrow C_{\lambda,m}$. The fact that $S_{m,\kappa}^{\lambda}$ is the image of $\cP_{m,m,\kappa}$ under $\chi$ and Corollary \ref{consofp} shows that $S_{m,\kappa}^{\lambda}$ is a constructible subset of $\pic^d(C)_{m,L}$. The assertion for $S_{i,\textbf{V}}^{\lambda}$ is proved similarly.
\end{proof}

\begin{lemma}\label{consofS}
$\widetilde{S}_{i,j,\textbf{V}}^\lambda$ is a constructible subset of the Grassmann bundle $Gr(\kappa_j,\mathcal{D}_i)$ on $C_{\lambda,i}$.
\end{lemma}
The proof of this lemma is similar to those of Lemma \ref{consofgf} and Corollary \ref{consofp}, hence we leave it to the reader.

\providecommand{\bysame}{\leavevmode \hbox \o3em
{\hrulefill}\thinspace}

\end{document}